%% file: fixedversion.tex
\documentclass[11pt,reqno]{amsart}
\usepackage{soul}
\usepackage{numprint} 
\usepackage[utf8]{inputenc}
\usepackage[all,cmtip]{xy}
\usepackage{amsmath,amsthm,amssymb}
\usepackage[ruled,vlined,linesnumbered]{algorithm2e}
\usepackage{verbatim}
\usepackage{colonequals}

\usepackage{tabularx}
\usepackage{booktabs}

\usepackage{graphics}
\usepackage[pagebackref=true, colorlinks, allcolors=blue]{hyperref}

\renewcommand*{\backref}[1]{}
\renewcommand*{\backrefalt}[4]{%
  \ifcase #1 %
    \relax
  \or
    $\uparrow$#2.%
  \else
    $\uparrow$#2.%
  \fi%
}
\usepackage{cleveref}
\usepackage[usenames, dvipsnames]{xcolor}
\usepackage{tikz}
\usetikzlibrary{cd}
\usepackage{fullpage}

\usepackage{mathtools}
\mathtoolsset{showonlyrefs}

\usepackage[square,sort,comma,numbers]{natbib}

\definecolor{darkblue}{rgb}{0.0,0.0,1}
\hypersetup{colorlinks,breaklinks,
  linkcolor=darkblue,urlcolor=darkblue,
anchorcolor=darkblue,citecolor=darkblue}

\theoremstyle{plain}
\newtheorem{theorem}{Theorem}
\newtheorem*{theorem*}{Theorem}
\newtheorem{lemma}[theorem]{Lemma}
\newtheorem*{lemma*}{Lemma}
\newtheorem{proposition}[theorem]{Proposition}
\newtheorem*{proposition*}{Proposition}
\newtheorem{corollary}[theorem]{Corollary}
\newtheorem*{corollary*}{Corollary}

\newtheorem{conjecture}[theorem]{Conjecture}
\newtheorem{question}[theorem]{Question}
\theoremstyle{definition}
\newtheorem{remark}[theorem]{Remark}
\newtheorem{definition}[theorem]{Definition}
\newtheorem{example}[theorem]{Example}

\DeclareMathOperator{\gon}{gon}

\renewcommand{\epsilon}{\varepsilon}
\renewcommand{\phi}{\varphi}
\renewcommand{\theta}{\vartheta}

\def\Alphabet{A,B,C,D,E,F,G,H,I,J,K,L,M,N,O,P,Q,R,S,T,U,V,W,X,Y,Z}
\def\alphabet{a,b,c,d,e,f,g,h,i,j,k,l,m,n,o,p,q,r,s,t,u,v,w,x,y,z}
\def\endpiece{xxx}
\def\makeAlphabet[#1]{\expandafter\makeA#1,xxx,}
\def\makealphabet[#1]{\expandafter\makea#1,xxx,}
\def\makeA#1,{\def\temp{#1}\ifx\temp\endpiece\else%
	\mkbb{#1}\mkfrak{#1}\mkbf{#1}\mkcal{#1}\mkscr{#1}\mkbs{#1}\expandafter\makeA\fi}%
\def\makea#1,{\def\temp{#1}\ifx\temp\endpiece\else\mkfrak{#1}\mkbf{#1}\mkbs{#1}\expandafter\makea\fi}%
\def\mkbb#1{\expandafter\def\csname bb#1\endcsname{\mathbb{#1}}}
\def\mkfrak#1{\expandafter\def\csname fr#1\endcsname{\mathfrak{#1}}}
\def\mkbf#1{\expandafter\def\csname b#1\endcsname{\mathbf{#1}}}
\def\mkcal#1{\expandafter\def\csname c#1\endcsname{\mathcal{#1}}}
\def\mkscr#1{\expandafter\def\csname s#1\endcsname{\mathscr{#1}}}
\def\mkbs#1{\expandafter\def\csname bs#1\endcsname{{\boldsymbol{#1}}}}
\def\makeop[#1]{\xmakeop#1,xxx,}
\def\mkop#1{\expandafter\def\csname #1\endcsname{{\mathrm{#1}}}} %
\def\xmakeop#1,{\def\temp{#1}\ifx\temp\endpiece\else\mkop{#1}\expandafter\xmakeop\fi}%
\def\makeup[#1]{\xmakeup#1,xxx,}
\def\mkup#1{\expandafter\def\csname #1\endcsname{{\mathrm{#1}\,}}} %
\def\xmakeup#1,{\def\temp{#1}\ifx\temp\endpiece\else\mkup{#1}\expandafter\xmakeup\fi}%
\makeAlphabet[\Alphabet]
\makealphabet[\alphabet]
\makeop[Hom,Aut,End,Mor,SL,GL,H,ord,Irr,Ell,Gal,Cl,Pic,NS,Gal,d,Re,Im,res,Symb,Ev,Char,Ram,SU]
\makeup[Spec,Proj,id,dR,new,old,AJ,tr,dim,ker,im,coker]

\DeclareMathOperator{\lcm}{lcm}

\newcommand{\Q}{\bQ}

\newcommand{\Z}{\bZ}
\newcommand{\PP}{\mathbf P}
\newcommand{\Jac}{\operatorname{Jac}}
\newcommand{\proj}{\operatorname{proj}}

\newcommand{\isom}{\cong}


\makeatletter
\let\@wraptoccontribs\wraptoccontribs
\makeatother

\title{Towards a Classification of Isolated $j$-invariants}
\input{author_info.tex}

\contrib[with an appendix by]{Maarten Derickx and Mark van Hoeij}

\begin{document}
\maketitle

\begin{abstract}
We develop an algorithm to test whether a non-CM elliptic curve $E/\Q$ gives rise to an isolated point of any degree on any modular curve of the form $X_1(N)$. This builds on prior work of Zywina which gives a method for computing the image of the adelic Galois representation associated to $E$. Running this algorithm on all elliptic curves presently in the $L$-functions and Modular Forms Database and the Stein--Watkins Database gives strong evidence for the conjecture that $E$ gives rise to an isolated point on $X_1(N)$ if and only if $j(E)=-140625/8, -9317,$ $351/4$, or $-162677523113838677$.
\end{abstract}

\section{Introduction}
\label{sec:intro}

The modular curve $X_1(N)$ is an algebraic curve over $\Q$ whose non-cuspidal points parametrize elliptic curves with a distinguished point of order $N$. We are interested in studying isolated points on $X_1(N)$, which are roughly those not belonging to an infinite family of points parametrized by a geometric object. For example, if $f\colon X_1(N) \rightarrow \PP^1$ is a rational map of degree $d$, then $f^{-1}(\PP^1(\Q))$ contains infinitely many closed points of degree $d$ by Hilbert's irreducibility theorem \cite[Chapter~9]{serre97}. We say a degree $d$ point \emph{not} arising from such a map is \textbf{$\PP^1$-isolated}. Other infinite families of degree $d$ points correspond to positive rank abelian subvarieties of the curve's Jacobian; see Section \ref{sec:background} for details. A point which is not thus parametrized is \textbf{AV-isolated}. If a closed point $x\in X_1(N)$ is both $\PP^1$- and AV-isolated, then we say $x$ is \textbf{isolated}. One special class of isolated points is \textbf{sporadic} points, which are points $x \in X_1(N)$ such that there are only finitely many points on $X_1(N)$ of degree at most $\deg(x)$. While every sporadic point is isolated \cite[Theorem 4.2]{BELOV}, the converse need not hold.

Elliptic curves with complex multiplication (CM) provide many natural examples
of isolated points, since the extra endomorphisms of a CM elliptic curve
constrain the size of the image of the associated Galois representation.
Indeed, as shown in \cite[Theorem~8.2]{CGPS2022}, there exist sporadic CM
points on $X_1(N)$ for all $N\geq 721$. Non-CM isolated points on $X_1(N)$ remain much more mysterious, and are tied to open uniformity problems of Balakrishnan and Mazur \cite[Conjecture 17]{BalakrishnanMazur23} and Serre \cite[$\S4.3$]{serre72}; see Section \ref{sec:UniformityConnections} for details. One recent line of investigation has focused on the class of isolated points associated to non-CM elliptic curves with $j$-invariant in $\Q$. Prior to this work, there were only three known examples of such elliptic curves, up to isomorphism over $\overline{\Q}$:
\begin{itemize}
    \item The elliptic curve with $j$-invariant $-140625/8$ corresponds to two sporadic points of degree 3 on $X_1(21)$. This example was first discovered by Najman \cite{najman16}. In fact, this is the unique elliptic curve giving a sporadic point of degree at most 3 on \emph{any} modular curve $X_1(N)$, as shown in \cite{DEvHMZB2021}.

    \item The elliptic curve with $j$-invariant $-9317$ gives three points of degree 6 on $X_1(37)$, as in work of van Hoeij \cite{vanHoeij}. Since 6 is less than half the $\Q$-gonality of $X_1(37)$, as computed in \cite{DerickxVanHoeij2014}, the points are necessarily sporadic by work of Frey \cite{frey}.

    \item The elliptic curve with $j$-invariant $351/4$ gives an isolated point of degree 9 on $X_1(28)$; see \cite[Theorem 2]{OddDeg}. There are infinitely many points on $X_1(28)$ of degree 9, as shown in \cite{DerickxVanHoeij2014}, so this point is isolated but not sporadic.
\end{itemize}

If $x \in X_1(N)$ is an isolated (resp. sporadic) point, we say $j(x) \in X_1(1) \cong \mathbf{P}^1$ is an \textbf{isolated} (resp. \textbf{sporadic}) \textbf{$j$-invariant}. Thus the three $j$-invariants listed above are isolated $j$-invariants, while only the first two, $-140625/8$ and $-9317$, are sporadic $j$-invariants. We have good reason to believe that the set of all isolated $j$-invariants in $\Q$ is finite. Indeed, in \cite[Corollary 1.7]{BELOV} the authors show that this would follow from an affirmative answer to Serre’s Uniformity Question \cite{serre72}, which is now a conjecture of Sutherland \cite{sutherland} and Zywina \cite{ZywinaImages}. 
 Moreover, in \cite[$\S1.2$]{BELOV}, the authors pose the following question:

\begin{question}[Bourdon, Ejder, Liu, Odumodu, Viray]
    Can one explicitly identify the (likely finite) set of isolated $j$-invariants in $\Q$?
\end{question}



This question serves as motivation for the present work. To this end, we develop an algorithm which can be used to determine whether a given non-CM $j$-invariant in $\Q$ is isolated. Starting with the image of the adelic Galois representation associated to $E/\Q$ with $j(E)=j$, as computed by Zywina \cite{ZywinaAlgorithm}, we apply various filters to determine whether there exists an isolated point $x \in X_1(N)$ with $j(x)=j$ for some $N$.

\begin{algorithm}[h!]
\renewcommand{\thealgocf}{} 
\caption{Main Algorithm} 
    \KwIn{A non-CM $j$-invariant $j \in \Q$.}
    \KwOut{A finite list $\{(a_1, d_1), \ldots, (a_k, d_k)\}$ of (level, degree) pairs such that $j$ is isolated (respectively, sporadic) if and only if there exists an isolated (respectively, sporadic) point $x \in X_1(a_i)$ of degree $d_i$ with $j(x) = j$ for some $(a_i, d_i)$ in the list.}
\end{algorithm}

\noindent
In particular, if the output of Algorithm \ref{alg:mainalgorithm} is the empty set, then $j$ is not an isolated $j$-invariant.
If the output is nonempty,
one can try to use other techniques to determine whether each point of degree $d_i$ on $X_1(a_i)$ associated to $E$ is isolated, from which we can definitively say whether $j(E)$ is an isolated $j$-invariant.

We ran Algorithm \ref{alg:mainalgorithm} on all elliptic curves currently in the $L$-functions and Modular Forms Database (LMFDB) \cite{LMFDB} and in the Stein--Watkins Database \cite{SteinWatkins}, which together contain over 36~million distinct non-CM $j$-invariants associated to elliptic curves over $\Q$ of conductor at most $10^8$. The output shows that all but 6 of the non-CM $j$-invariants included in these databases are \emph{not} isolated. As noted above, half of these remaining $j$-invariants are known to be isolated; in Section~\ref{sec:remainingfilters} we perform a case-by-case analysis on the remaining $3$ candidates. In particular, these findings imply the following result.

\begin{theorem}\label{LMFDBoutputThm}
Let $x=[E,P]\in X_1(N)$ be a non-CM isolated point with $j(E) \in \Q$.
Fix an equation for $E/\Q$ and let $N_E$ denote its conductor.
Suppose that one of the following holds:
\begin{itemize}
    \item $N_E \leq \numprint{500000}$,
    \item $N_E$ is only divisible by primes $p \leq 7$, or
    \item $N_E=p \leq \numprint{300000000}$ for some prime number $p$.
\end{itemize}
Then $j(E) \in \{-140625/8,-9317,351/4, -162677523113838677\}$. Moreover, each one of these $j$-invariants corresponds to a $\PP^1$-isolated point on $X_1(21)$, $X_1(37)$, $X_1(28)$, or $X_1(37)$, respectively.
\end{theorem}

\begin{remark}
Though we did not find a result in the literature showing the degree 18 point on $X_1(37)$ associated to $j=-162677523113838677$ is $\PP^1$-isolated, the point itself was well-known prior to this work. Indeed, this $j$-invariant corresponds to one of two non-CM elliptic curves over $\Q$ with a rational cyclic 37-isogeny (see, e.g., \cite[Table 4]{LRAnn}), and any elliptic curve with a $\Q$-rational cyclic $N$-isogeny will give a point on $X_1(N)$ in degree at most $\varphi(N)/2$. Theorem \ref{DisolatedThm} in the appendix by Maarten Derickx and Mark van Hoeij shows that this point is AV-isolated as well (answering a question of Ejder \cite[Remark 1.3]{Ejder}). This allows us to conclude that the 4 $j$-invariants identified in Theorem \ref{LMFDBoutputThm} are in fact \emph{isolated}.
\end{remark}

We conjecture that the four non-CM $j$-invariants identified above are the \emph{only} $j$-invariants in $\Q$ associated to non-CM elliptic curves which give rise to isolated points on $X_1(N)$. Such a result has already been established for points of odd degree: by \cite{OddDeg} we know that $j=-140625/8$ and $j=351/4$ are the only non-CM $j$-invariants in $\Q$ giving an isolated point of odd degree on $X_1(N)$, even as $N$ ranges over all positive integers. Moreover, \cite{Ejder} shows that if $x \in X_1(\ell^n)$ is a non-CM isolated point with $\ell>7$ prime and $j(x)\in \Q$, then $j=-9317$ or $j=-162677523113838677$.

\begin{conjecture}\label{conj}
If $x\in X_1(N)$ is an isolated point with $j(x) \in \Q$, then $j(x)=-140625/8$, $-9317$, $351/4$, $-162677523113838677$, or one of the 13 CM $j$-invariants in $\Q$.
\end{conjecture}

\noindent Since any CM elliptic curve is known to produce sporadic points on infinitely many modular curves of the form $X_1(N)$ by \cite[Theorem 7.1]{BELOV}, it follows conversely that every $j$-invariant in this set is $\PP^1$-isolated (and in fact isolated --- see the appendix).

\begin{remark}
Let $x\in X_1(N)$ be an isolated point with $j(x) \in \Q$. One expects that the square-free part of the conductor of any $E$ with $j(E) = j(x)$ will be very small. The reason for this is that any such isolated point will generally have small mod $\ell$ image for some $\ell$ dividing $N$. Unless $\ell$ is small, we expect this to force potentially good reduction on $E$ at all odd primes; see~\cite[Corollary 4.4]{Mazur1978} and \cite[Theorem 5.1]{BiluParent11} for examples of this phenomenon. Indeed this happens for all the curves appearing in~\Cref{conj}: their conductors are either a square, or twice a square.

Heuristically one might except to find most or all of the isolated points among elliptic curves with relatively small conductor (and hence in the LMFDB). One can also numerically observe that elliptic curves with small Galois representations are over-represented among curves with small conductor. Only 4 of the first 50 curves ordered by conductor have trivial torsion groups and all of them have non-trivial isogenies, while one expects asymptotically almost all curves to have surjective Galois representations when ordered by height; see \cite[Theorem 1]{Duke97}.

These observations bolster the computational evidence supporting \Cref{conj}.
\end{remark}

It is natural to suspect a connection between isolated points on $X_1(N)$ and isolated points on $X_0(N)$, the modular curve parametrizing elliptic curves with a rational cyclic $N$-isogeny. We say a point $x \in X_0(N)(\Q)$ is \textbf{exceptional} if $X_0(N)(\Q)$ is finite and $x$ corresponds to a non-CM elliptic curve over $\Q$. It is worth noting that the sporadic points on $X_1(21)$ associated to $j=-140625/8$ and the sporadic points on $X_1(37)$ associated to $j=-9317$ lie above exceptional rational points on $X_0(21)$ and $X_0(37)$, respectively. One might wonder whether other sporadic $j$-invariants can be obtained by a similar construction. Running Algorithm~\ref{alg:mainalgorithm} on all 14 $j$-invariants corresponding to exceptional rational points on $X_0(N)$ for any $N$ (described, for example, in~\cite[Table 4]{LRAnn}) shows that there are no additional sporadic $j$-invariants. In particular, our algorithm shows that $j= -162677523113838677$ is not a sporadic $j$-invariant; see Example \ref{Ex45} for details.

\begin{theorem}
Let $E$ be an elliptic curve corresponding to an exceptional rational point on $X_0(N)$ for some positive integer $N$. If $j(E)$ is sporadic, then $j(E)=-140625/8$ or $-9317$.
\end{theorem}

It is still an open problem to determine all sporadic points $x \in \cup_{N \in \Z^+} X_1(N)$ with $j(x) =-140625/8$ or $-9317$.

\subsection{Key Components of Algorithm.} The first step of our algorithm applies results of \cite{BELOV} and \cite{ZywinaAlgorithm} to compute the finite set of \textbf{primitive points} associated to a non-CM elliptic curve $E/\Q$. The primitive points are characterized by the following theorem. 
\begin{theorem}\label{Thm:PrimPtsIntro}
Let $E/\Q$ be a non-CM elliptic curve.  There exists a finite set $\mathcal{P}=\mathcal{P}(E)$ of primitive points in $\cup_{n \in \Z^+} X_1(n)$ associated to $E$ which are characterized by the following properties:
\begin{enumerate}
\item For each $N\in \Z^+$, a point $x \in X_1(N)$ with $j(x)=j(E)$ corresponds to a unique element $x' \in \mathcal{P}$ under the natural projection map. Moreover, if $x' \in X_1(a)$, then $a \mid N$ and $\deg(x)=\deg(f)\cdot \deg(x')$, where $f\colon X_1(N) \rightarrow X_1(a)$ is the natural map.
\item The rational number $j(E)$ is isolated (respectively, sporadic) if and only if there exists an isolated (respectively, sporadic) point in $\mathcal{P}$.
\end{enumerate}
Moreover, the set $\mathcal{P}$ is minimal with respect to condition (i).
\end{theorem} 

\noindent By Theorem~\ref{Thm:PrimPtsIntro}\,(i), one can think of $\mathcal{P}(E)$ as the minimal set needed to reproduce the degrees of all points $x \in \cup_{n \in \Z^+} X_1(n)$ with $j(x) = j(E)$. Moreover, any isolated point $x \in X_1(n)$ with $j(x)=j(E)$ will correspond to a unique primitive point of minimal level. See Section \ref{sec:PrimitivePoints} for details.

The second part of Algorithm \ref{alg:mainalgorithm} works to show the primitive points corresponding to $E$ are \emph{not} isolated. For example, if the Riemann--Roch space associated to $x \in X_1(N)$ has dimension at least 2, then $x$ is not $\PP^1$-isolated (and therefore not isolated). In other cases, we can show $x \in X_1(N)$ is not isolated by applying the following result.

\begin{theorem}\label{Prop:Genus0Intro}
Let $E/\Q$ be an elliptic curve, and let $H \leq \GL_2(\Z/N\Z)$ be the image of the mod $N$ Galois representation of $E$, after some choice of basis. If the modular curve $X_H$ has genus $0$, then there are no isolated points on $X_1(N)$ associated to $E$.
\end{theorem}

\noindent In particular, any elliptic curve with adelic image of genus 0 does not give rise to any isolated points on $X_1(N)$, even as $N$ ranges over all positive integers. However, Theorem \ref{Prop:Genus0Intro} is more broadly applicable. Even when the adelic image of $E$ has positive genus, it can still be that for all \emph{primitive} points $x \in X_1(a)$ associated to $E$, the mod $a$ Galois representation of $E$ gives a genus 0 modular curve. This occurs more often than one might expect: our preliminary computations identified at least 89 distinct such non-CM elliptic curves (up to $\overline{\Q}$-isomorphism) just within those currently in the LMFDB. See Example \ref{eg:modngenus0} for one such curve.

\subsection{Connection to Open Uniformity Problems}\label{sec:UniformityConnections} Several well-known uniformity problems can be tied to isolated points on $X_1(N)$. One of the most longstanding examples is Serre's Uniformity Problem \cite[\S\,4.3]{serre72}, which in modern formulations \cite{sutherland,ZywinaImages} asks whether the mod $\ell$ Galois representation for any non-CM elliptic curve over $\Q$ is surjective for all $\ell>37$. In the proof of \cite[Theorem~1.3]{BourdonNajman2021}, the authors show that a non-CM elliptic curve $E/\Q$ with non-surjective mod $\ell$ Galois representation can be used to construct sporadic points on $X_1(\ell^2)$ for all $\ell$ sufficiently large. This approach allows one to phrase Serre's Uniformity Problem to be about controlling isolated points on $X_1(\ell^2)$ within certain families of non-CM $\Q$-curves.

A more recent example is \cite[Conjecture 17]{BalakrishnanMazur23}, where Balakrishnan and Mazur conjecture that for sufficiently large $N$, any elliptic curve giving a quadratic point on $X_0(N)$ must have complex multiplication. Since a quadratic point on $X_0(N)$ will give a sporadic point on $X_1(N)$ for $N$ sufficiently large, we can connect this conjecture to one about non-CM isolated points $x \in X_1(N)$ with $j(x)$ generating at most a quadratic extension.

\subsection{Outline}
After providing relevant background material in Section \ref{sec:background}, we give an overview of the main algorithm in Section \ref{sec:mainalg}. The sub-algorithms used to compute primitive points and related mathematical results are discussed in Sections \ref{sec:madicrep}--\ref{sec:primitivedeg}, with the proof of Theorem \ref{Thm:PrimPtsIntro} appearing in Section \ref{sec:PrimitivePoints}. Results on genus 0 adelic images, including the proof of Theorem \ref{Prop:Genus0Intro}, are in Section \ref{sec:genus0}.  We address the validity of Algorithm \ref{alg:mainalgorithm} in Section \ref{sec:validity}. The output the main algorithm obtained after running it on elliptic curves in the LMFDB and the Stein--Watkins database is discussed in Section \ref{sec:remainingfilters}, along with its final analysis.

\subsection{Code} We have implemented our algorithm in Magma \cite{Magma}. Code is available in the GitHub repository at \url{https://github.com/davidlowryduda/isolated_points}.

\section*{Acknowledgments}
We thank Pete Clark, Maarten Derickx, Jeremy Rouse, Andrew Sutherland, and David Zureick-Brown for helpful conversations. We also thank David Zywina for making his code to compute Galois images available, and David Roe for further making Zywina's code and implementation more readily available for other mathematicians to use; this project would not have been possible without their work. This project began at the \emph{COmputations and their Uses in Number Theory} conference at CIRM in March 2023; we thank the organizers as well as CIRM for providing the opportunity for collaboration. We are very grateful for the suggestions and feedback of the anonymous referee.

AB was supported by NSF grants DMS-2145270 and DMS-1928930. Part of the work was completed while this author was in residence at the Simons Laufer Mathematical Science Institute in Berkeley, CA, during the semester of Diophantine Geometry. TK was partially supported by the 2021 MSCA Postdoctoral Fellowship 01064790 -- ExplicitRatPoints. DLD was supported by the Simons Collaboration in Arithmetic Geometry, Number Theory, and Computation via the Simons Foundation grant 546235. TM was partially supported by the Commonwealth Cyber Initiative. FN is supported by the QuantiXLie Centre of Excellence, a
  project co-financed by the Croatian Government and European Union
  through the European Regional Development Fund - the Competitiveness
  and Cohesion Operational Programme (Grant KK.01.1.1.01.0004) and by the Croatian
Science Foundation under the project no. IP-2022-10-5008. HS is supported by the DFG-grant STO 299/17-1.
MvH was supported by NSF grant 2007959.

\section{Background}
\label{sec:background}
\subsection{Isolated Points on Curves} \label{sec:2.1}
Let $C$ be a curve, by which we will mean a smooth projective geometrically integral $1$-dimensional scheme defined over a number field $k$; we suppose all curves satisfy these assumptions throughout the paper. To streamline our exposition, we assume there exists a point $P_0 \in C(k)$, but this is not required; see \cite[$\S4$]{BELOV}. Throughout, we consider closed points of the curve $C$, which correspond to $\Gal_k$-orbits of points in $C(\overline{k})$. The degree of $x$ is defined to be the degree of the residue field $k(x)$ over $k$, or alternatively, to be the length of the $\Gal_k$-orbit of points in $C(\overline{k})$ corresponding to $x$.

To any closed point $x \in C$ of degree $d$ we associate the $k$-rational effective divisor
\[
P_1+ \cdots +P_{d}, 
\] 
where $P_1, \dots, P_{d}$ are the points in the $\Gal_{k}$-orbit associated to $x$. With this identification, we can study the image of $x$ under the natural map from the $d$th symmetric power of $C$ to the curve's Jacobian
\begin{align}
\label{eqn:phid}
\Phi_d\colon C^{(d)} \rightarrow \Jac(C)
\end{align}
which sends an effective divisor $D$ of degree $d$ to the class $[D - dP_0]$. If $\Phi_d(x)=\Phi_d(y)$ for some other point $y \in C^{(d)}(k)$, then there exists a non-constant function $f \in k(C)^{\times}$ with $\text{div}(f)=x-y$. Since $x$ is a degree $d$ point and $x\not=y$, the divisors associated to $x$ and $y$ have distinct support so $f\colon C \rightarrow \mathbf{P}^1_{k}$ is a dominant morphism of degree $d$.\footnote{In particular, this shows $\Phi_d$ is injective if $d$ is less than the $k$-gonality of $C$, which is the least degree of a non-constant rational map to $\mathbf{P}^1$.} By Hilbert's irreducibility theorem \cite[Chapter 9]{serre97}, $f^{-1}(\mathbf{P}^{1}(k))$ will contain infinitely many closed points of degree $d$. On the other hand, if $\Phi_d$ is injective on closed points of degree $d$, then Faltings's Theorem \cite{faltings} implies that all but finitely many such points are parametrized by translates of positive rank abelian subvarieties of $\Jac(C)$. This inspires the following:
\begin{definition} Let $C$ be a curve defined over a number field $k$. Let $\Phi_d$ be the map in \eqref{eqn:phid}.
\begin{enumerate}
\item A closed point $x \in C$ of degree $d$ is $\PP^1$\textbf{-parametrized} if there exists a point $x' \in C^{(d)}(k)$ with $x' \neq x$ such that $\Phi_d(x)=\Phi_d(x')$. Otherwise, we say $x$ is $\PP^1$\textbf{-isolated}.
\item A closed point $x \in C$ of degree $d$ is \textbf{AV-parametrized} if there exists a positive rank abelian subvariety $A/k$ with  $A \subset \Jac(C)$ such that $\Phi_d(x)+A \subset \im(\Phi_d)$. Otherwise, we say $x$ is \textbf{AV-isolated}.
\item A closed point $x \in C$ of degree $d$ is \textbf{isolated} if it is both $\PP^1$-isolated and AV-isolated.
\item A closed point $x \in C$ of degree $d$ is \textbf{sporadic} if there are only finitely many closed points of $C$ of degree at most $\deg(x)$.
\end{enumerate}
\end{definition}

If $C$ has genus $g\geq 2$, then the collection of all points on $C$ with coordinates in $k$ is finite by Faltings's theorem \cite{faltings83}. In general, the set $C(k)$ sits inside a larger finite set of points, namely, the set of all isolated points of $C$.
    \begin{theorem}[{Bourdon, Ejder, Liu, Odumodu, Viray, \cite[Theorem 4.2]{BELOV}}]\label{thm:FiniteIsolated}
        Let $C$ be a curve over a number field.
        \begin{enumerate}
            \item There are infinitely many degree $d$ points on $C$ if and only if there is a degree $d$ point on $C$ that is \emph{not} isolated.
            \item There are only finitely many isolated points on $C$.
        \end{enumerate}
    \end{theorem}
It follows from Theorem \ref{thm:FiniteIsolated} that every sporadic point is isolated, but the converse need not hold.

\subsection{Modular Curves} \label{sec:modular_curves}
For any subgroup $H \leq \GL_2(\Z/N\Z)$, we define the modular curve $X_H$ to be the coarse space of the stack $\mathcal{M}_H$, as defined in Deligne--Rapoport \cite{DR}. The curve $X_H$ is a scheme over $\text{Spec }\Z[1/N]$ and parametrizes generalized elliptic curves with $H$-level structure. In particular, its $k$-rational points roughly classify elliptic curves over $k$ whose mod $N$ image is contained in $H$; see, for example, \cite[$\S2.3$]{RSZB2022} for details. If we take
\[
B_1(N) =\left\{ \begin{pmatrix}
1 & * \\
0 & *
\end{pmatrix} \right \} \leq \GL_2(\Z/N\Z),
\]
then $X_{B_1(N)}=X_1(N)$, the modular curve whose noncuspidal points parametrize elliptic curves with a distinguished point of order $N$. There is an analytic isomorphism between $X_1(N)(\mathbf{C})$ and the Riemann surface constructed as a quotient of the extended upper-half plane by the congruence subgroup
\[
\Gamma_1(N) \coloneqq\left\{ \begin{pmatrix}
a & b \\
c & d
\end{pmatrix} \in \SL_2(\Z) \, : \, c \equiv 0 \pmod{N} \text{ and } a \equiv d \equiv 1 \pmod{N} \right \},
\]
with matrices acting via linear fractional transformations. If $N \geq 4$, then $\mathcal{M}_{B_1(N)}$ is its own coarse moduli space, and so noncuspidal $k$-rational points of $X_1(N)$ classify pairs $(E,P)$, where $E/k$ is an elliptic curve and $P\in E(k)$, up to $k$-isomorphism.

We may also define a modular curve associated to an open subgroup $G$ of $\GL_2(\widehat{\Z})$. For any $N \in \Z^+$, let $\pi\colon\GL_2(\widehat{\Z}) \rightarrow \GL_2(\Z/N\Z)$ be the natural projection map, and define $G(N) \coloneqq \pi(G)$. We say $G$ has \textbf{level} $N$ if $G=\pi^{-1}(G(N))$ and $N$ is minimal with respect to this property. If $\det(G)=\widehat{\Z}^{\times}$, we define the modular curve $X_G \coloneqq X_{G(N)}$ where $N$ is the level of $G$. If $G=\GL_2(\widehat{\Z})$, then we identify $X_G \cong \PP^1$ with the $j$-line.

\subsection{Closed Points on Modular Curves} \label{sec:ClosedPts} To discuss isolated points on modular curves, we must consider closed points on $X_1(N)$, viewed always as a scheme over $\Q$. Let $k$ be a field with an embedding of $k$ into $\overline \Q$. Given an elliptic curve $E/k$ with $P \in E(k)$ of order $N$, the pair $(E,P)$ induces a $k$-valued point on $X_1(N)$ via the moduli interpretation described above. We denote this $k$-valued point by $(E,P)_k$, and by definition it corresponds to a morphism of $\Q$-schemes $f\colon\text{Spec } k \rightarrow X_1(N)$. The map $f$ sends the unique point of $\text{Spec } k$ to a point $x \in X_1(N)$, and we call $x$ the \textbf{closed point} associated to $(E,P)$.\footnote{Note $x$ is indeed closed since $\Q(x)/\Q$ is finite; see, for example, \cite[Exercise~5.9, p.~76]{Liu2002}.} We define the \textbf{degree} of $x$ to be the degree of the residue field $\Q(x)$ over $\Q$. Since there are many scheme-valued points which induce the same closed point, it is sometimes preferable to consider Galois orbits of points in $X_1(N)(\overline{\Q})$, which are in bijection with the set of closed points. Thus one could alternatively define the closed point associated to $(E,P)$ as the $\Gal_\Q$-orbit of $(E,P)_{\overline{\Q}}$.

\begin{remark} Given a $k$-valued point $(E,P)_k$, we note that the degree of the associated closed point $x$ may be strictly less than the degree of $k$. However, there always exists $E'/\Q(x)$ with $j(E')=j(E)$ and $P' \in E'(\Q(x))$, where the point $P' \in E'$ maps to $P \in E$ under a $\overline \Q$-isomorphism sending $E'$ to $E$. See \cite[p.~274, Proposition~VI.3.2]{DR}. The pair $(E',P')$ gives a $\Q(x)$-valued point, and it is the unique $\Q(x)$-valued point such that $(E,P)_{k}=(E',P')_{k}.$
\end{remark}

Let $E/k$ be an elliptic curve and let $P \in E(k)$ a point of order $N$.
For any $\xi \in \Aut(E)$, the pair $(E,\xi P)$ induces the same closed point $x \in X_1(N)$, since $\xi$ provides the necessary isomorphism. This can be used to obtain a more explicit description of the residue field $\Q(x)$.

\begin{lemma}\label{ResidueFieldLemma}
Let $E$ be an elliptic curve defined over $\Q(j(E))$, and let $P \in E$ be a point of order $N$. Then the residue field of the closed point $x \in X_1(N)$ associated to $(E,P)$ is given by
\[
\Q(x)\cong\Q(j(E),\mathfrak{h}(P)),
\]
where $\mathfrak{h} \rightarrow E/\Aut(E) \cong \mathbf{P}^1$ is a Weber function for $E$.
\end{lemma}

\begin{proof}
See, for example, \cite[Lemma 2.5]{BourdonNajman2021}.
\end{proof}

 \begin{remark}\label{ResidueFieldRmk}At times, it can be useful to work with an explicit model-independent formulation of a Weber function. For example, if $E:y^2=4x^3-c_2x-c_3$ and $j(E)\neq 0,1728$, we can define
\[
\mathfrak{h}((x,y))=\frac{c_2c_3}{\Delta}x,
\]
where $\Delta=c_2^3-27c_3^2$. One can verify:
\begin{enumerate}
    \item We have $\mathfrak{h}(P)=\mathfrak{h}(P')$ if and only if $P=\xi P'$ for some $\xi \in \Aut(E)$.
    \item If $\eta \colon E \rightarrow E'$ is an isomorphism, then $\mathfrak{h}_E=\mathfrak{h}_{E'}\circ \eta$.
\end{enumerate}
See \cite[p. 107]{shimura}. In particular, if $E/\Q(j(E))$ does not have complex multiplication and $P=(x_0,y_0)$, then one can take $\Q(x)\cong\Q(j(E),x_0)$. 
\end{remark}

\begin{example}[Closed points versus geometric points I]
Let $E_1\colon y^2+xy+y=x^3-x^2-3x+3$ and $P_1=(-1,-2)$ be a point of order 7. Then $(E_1,P_1)$ gives a $\Q$-valued point on $X_1(7)/\Q$ and also a closed point $x\in X_1(7)$ of degree 1. On the other hand, let $E_2\colon y^2=x^3-43x-166$ and $P_2=(5,\sqrt{-256})$ be a point of order 7. Then $(E_2,P_2)$ gives a $\Q(\sqrt{-256})$-valued point on $X_1(7)/\Q$ and also a closed point $x\in X_1(7)$ of degree 1 by Remark \ref{ResidueFieldRmk}. However, both $(E_1,P_1)$ and $(E_2,P_2)$ induce the same geometric point on $X_1(7)$ since $(E_1,P_1)_{\overline{\Q}}=(E_2,P_2)_{\overline{\Q}}$.

In fact, $(E_1,P_1)_{{\Q(\sqrt{-256})}} =(E_2,P_2)_{{\Q(\sqrt{-256})}}$. This association is made naturally via the Kubert--Tate normal form associated to $E_2/\Q(\sqrt{-256})$, with $P_2=(5,\sqrt{-256})$:
\[
    E_3\colon y^2-xy-4y=x^3-4x^2,\, \, P_3=(0,0).
\]
We can check that $(E_1,P_1)_{\Q}=(E_3,P_3)_{\Q}$ and $(E_2,P_2)_{{\Q(\sqrt{-256})}} =(E_3,P_3)_{{\Q(\sqrt{-256})}}$. Thus it is fair to say that $(E_2,P_2)$ induces a $\Q$-valued point, even though $P_2$ is not defined over $\Q$.
\end{example}

\begin{example}[Closed points versus geometric points II]
The distinction between closed points and geometric points can be seen when counting the number of points of a particular degree. For example, let $E$ be the elliptic curve with LMFDB label \href{https://www.lmfdb.org/EllipticCurve/Q/162/c/3}{162.c3}. Then $E$ possesses a unique $\Q$-rational subgroup of order 21 with generator $P$. For each $a \in (\Z/21\Z)^{\times}$, we consider the geometric point on $X_1(21)$ associated to $(E,aP)$. Since $(E,aP)_{\overline{\Q}}=(E,-aP)_{\overline{\Q}}$, we find that there are six distinct geometric points corresponding to $(E,aP)$ for $a \in (\Z/21\Z)^{\times}$. However, these six $\overline{\Q}$-points lie in two Galois orbits, each of size 3. Thus there are two closed points of degree 3, and the cardinality of the Galois orbit equals the degree.
\end{example}


\subsection{Maps Between Modular Curves}

\begin{proposition} \label{prop:mapdegree}
If $G \subseteq G' \subseteq \GL_2(\widehat{\Z})$ are two open subgroups with surjective determinant, then there is a natural $\Q$-rational morphism $X_G \rightarrow X_{G'}$ of degree $[\pm G': \pm G]$. Here, $\pm G$ denotes the subgroup generated by $G$ and $-I_2$.
\end{proposition}
\begin{proof}
Let $N$ be the level of $G$. For any subgroup $H$ of $\GL_2(\Z/N\Z)$, we have $\Q(X_H)=\Q(X(N))^H$. Now by Galois theory it follows that $\Q(X_G) \supseteq \Q(X_{G'})$, so we conclude that there exists a $\Q$-rational morphism $f\colon X_G \rightarrow X_{G'}$.

To determine its degree, let $\Gamma$ and $\Gamma'$ be the intersection with $\SL_2(\Z)$ of the inverse image of $G$ and $G'$ in $\GL_2(\Z)$. Over $\bC$, the morphism $f$ is the quotient map $\Gamma \backslash \mathcal{H}^* \rightarrow \Gamma' \backslash \mathcal{H}^*$. Here, $\mathcal{H}^*$ is the extended complex upper half plane. Since the kernel of the action of $\SL_2(\Z)$ on $\mathcal H$ is $\pm I$, the degree of $f$ is as claimed.
\end{proof}

Let $a$ and $b$ be positive integers. Taking $G=\pi^{-1}(B_1(ab))$ and $G'=\pi^{-1}(B_1(a))$ gives the following corollary, which under the moduli interpretation corresponds to sending $(E,P)$ to $(E,bP)$.  

\begin{corollary}\label{Cor:DegreeFormula}
                For positive integers $a$ and $b$, the natural $\Q$-rational map $f\colon X_1(ab) \rightarrow X_1(a)$ has
                \[
                    \deg(f)=
                    c_{f}\cdot b^2 \prod_{p \mid b,\, p \nmid a}
                    \left(1-\frac{1}{p^2}\right),
                \]
                where $c_{f}=\frac{1}{2}$ if $a \leq 2$ and $ab>2$, and $c_{f}=1$ otherwise.
\end{corollary}

\subsection{Galois Representations} If $E$ is an elliptic curve defined over a number field $k$, then $\Gal_k$ acts naturally on the $\overline{k}$-points of $E$. On torsion points, this action is described by the \textbf{adelic Galois representation} associated to $E/k$,
\[
\rho_{E}\colon \Gal_k \rightarrow \Aut(E(\overline{k})_{\text{tors}}) \cong \GL_2(\widehat{\Z}).
\]
From this we can obtain two other Galois representations. On the one hand, fixing a positive integer $m$, we can choose to record the action of $\Gal_k$ on points whose order is divisible only by those primes dividing $m$. This is the \textbf{$m$-adic Galois representation} associated to $E$:
\[
\rho_{E,m^{\infty}}\colon \Gal_k \xrightarrow{\rho_E} \GL_2(\widehat{\Z}) \cong  \prod_{p \text{ prime}} \GL_2(\Z_{p}) \xrightarrow{\proj} \prod_{p \mid m} \GL_2(\Z_{p}).
\]
In particular, if $m=\ell$ is a prime number, we recover the standard $\ell$-adic representation associated to $E$. Alternatively, we may wish to record the Galois action only on points of order dividing $m$. We use $E[m]$ to denote the finite subgroup of such points. This gives the \textbf{mod $m$ Galois representation} associated to $E$,
\[
\rho_{E,m}\colon \Gal_k \rightarrow \Aut(E[m]) \cong \GL_2(\Z/m\Z).
\]
Note that $\rho_{E,m}$ agrees with the reduction of $\rho_E$ mod $m$.

If $E/k$ is a non-CM elliptic curve, we define the \textbf{level} of $\rho_E$ to be the smallest positive integer $N$ such that $\im \rho_E=\pi^{-1}(\im \rho_{E,N})$, where $\pi\colon\GL_2(\widehat{\Z}) \rightarrow \GL_2(\Z/N\Z)$ is the natural reduction map; that such an $N$ exists is a consequence of Serre's Open Image Theorem \cite{serre72}. The level of $\rho_{E,m^{\infty}}$ may be defined in an analogous way. We take the convention that $\GL_2(\Z/1\Z)$ denotes the trivial group, so level 1 corresponds to the associated Galois representation being surjective. Though $\rho_E$ is never surjective when $k=\Q$, this can occur for elliptic curves defined over number fields of higher degree \cite[Theorem 1.2]{Greicius2010}.

\section{Overview of the Main Algorithm}
\label{sec:mainalg}
The following algorithm is the main procedure for determining whether a given non-CM $j$-invariant in $\Q$ is the image of an isolated point on $X_1(N)$ under the map to the $j$-line, based on results in \cite{BELOV,ZywinaAlgorithm}. We note that for any CM $j$-invariant $j$, there exist infinitely many $N\in\Z^+$ for which there is a sporadic point $x\in X_1(N)$ with $j(x)=j$ by \cite[Theorem 7.1]{BELOV}, so it is not necessary to consider them in this algorithm. The outline below gives an overview of the structure, while the algorithms to perform particular steps are described in detail in Sections \ref{sec:madicrep}, \ref{sec:primitivedeg}, and \ref{sec:genus0images}. We will prove a theorem on the validity of Algorithm \ref{alg:mainalgorithm} in Section \ref{sec:validity}.

\setcounter{algocf}{0}
\begin{algorithm}[h!]
\caption{Main Algorithm}\label{alg:mainalgorithm}
\KwIn{A non-CM $j$-invariant $j \in \Q$.}
\KwOut{A finite list $\{(a_1, d_1), \ldots, (a_k, d_k)\}$ of (level, degree) pairs such that $j$ is isolated (respectively, sporadic) if and only if there exists an isolated (respectively, sporadic) point $x \in X_1(a_i)$ of degree $d_i$ with $j(x) = j$ for some $(a_i, d_i)$ in the list.}
Construct an elliptic curve $E/\Q$ with $j(E)=j$.\\
Compute the adelic image $G$ of $E/\Q$ as a subgroup of $\GL_2(\widehat{\Z})$ using Zywina's algorithm \cite{ZywinaAlgorithm}. Represent the output as the level $N$ and the subgroup $G(N)$ of $\GL_2(\Z/N\Z)$.\\
Apply Algorithm \ref{alg:reduce_level} to $G(N)$ to obtain the level $m_0$ of the $m$-adic Galois representation associated to $E$, where $m$ is the product of 2, 3, and all non-surjective primes.\\
Apply Algorithm \ref{alg:compute_prim_degree} to $\im \rho_{E,m_0}$. For each positive divisor $n$ of $m_0$ and each closed point $x \in X_1(n)$ with $j(x)=j$, this gives the primitive point associated to $x$, say $x' \in X_1(a)$ of degree $d$. Return a multiset $D$ with entries $\langle n, (a, d) \rangle$.\\
Construct the multiset $D' \subseteq D$ containing only those elements $\langle n, (a, d)\rangle$ for which $d\leq \text{genus}(X_1(a))$.\\
Create the multiset $M$ consisting of all pairs $(a,d)$ with $\langle n, (a,d) \rangle$ appearing in $D'$. We include $(a,d) $ with multiplicity $\mu$ if and only if $X_1(a)$ has $\mu$ distinct closed points of degree $d$ associated to $E$.\\
Remove from $M$ any pair $(a_i,d_i)$ where the mod $a_i$ Galois representation of $E/\Q$ corresponds to a modular curve of genus 0.\\
\Return{$M$} 

\end{algorithm}

In particular, note that if Algorithm \ref{alg:mainalgorithm} outputs $\{\}$, then $j$ is not the image of any isolated point on $X_1(N)$, even as $N$ ranges over all positive integers. See Corollary \ref{cor:P1_isolated}.

\begin{example}
If $j=-9317$, then Algorithm \ref{alg:mainalgorithm} returns $ \{(37,6) ^3\}$. This means that any isolated point $x\in X_1(N)$ with $j(x)=j$ and $N \in \Z^+$ maps down under the natural projection map to one of the 3 closed points of degree 6 on $X_1(37)$. In fact, these points are all sporadic by \cite[Proposition 2]{frey}, since $6 < \frac{1}{2} \gon_{\Q}(X_1(37)) = 18$. Here, the gonality computation is a result of \cite{DerickxVanHoeij2014}. Thus $j$ is a sporadic --- and hence isolated --- $j$-invariant.
\end{example}

\begin{example}
If $j=-121$, then Algorithm \ref{alg:mainalgorithm} returns $\{\}$. This means that there are no isolated points on $X_1(N)$ associated to this $j$-invariant.
\end{example}

\begin{example}
If $j=-882216989/131072$, then Algorithm \ref{alg:mainalgorithm} returns  $\{( 17,4) ^2\}$. This means that $j$ is associated to an isolated point on a modular curve of the form $X_1(N)$ if and only if there exists a degree 4 isolated point $x \in X_1(17)$ with $j(x)=j$. In Section \ref{sec:remainingfilters}, we will see that no such isolated point exists, from which we can conclude that $j$ is not an isolated $j$-invariant.
\end{example}




\section{\texorpdfstring{Computing the level of the $m$-adic representation}{Computing the level of the m-adic representation}}
\label{sec:madicrep}

Given a non-CM elliptic curve $E/\Q$, we define the set
\[
S_E\coloneqq \{2,3\} \cup \{ \ell: \rho_{E,\ell^{\infty}}(\Gal_{\Q}) \neq \GL_2(\Z_{\ell}) \}.
\] Here, we include 2 and 3 in $S_E$ to allow us to apply results of \cite{BELOV} in later steps of the algorithm; see Section \ref{ssec:finiteness of PE} for details.
For $m \coloneqq \prod_{\ell \in S_E} \ell$, we want an algorithm to obtain the level of the $m$-adic Galois representation associated to $E$ from the image of the adelic representation of $E$, where the latter can be computed by Zywina's algorithm \cite{ZywinaAlgorithm}. If $\im \rho_E$ has level $N$, define $n=\prod_{\ell \in S_E} \ell^{v_{\ell}(N)}$ and let $m_0$ denote the level of $\rho_{E,m^{\infty}}$. We will show in Proposition \ref{Prop:Level} that $m_0| n$ and $n | N$. Each of these divisibilities can be proper, as the following examples illustrates.

\begin{example}
 If $E=75072.bc2$, we see that $\rho_E$ has level $N=4692=2^2 \cdot 3 \cdot 17 \cdot 23$. Since 2 is the only non-surjective prime, $n=2^2\cdot 3$. However, the level of the $6$-adic Galois representation associated to $E$ is 2. This reduction in level is due to the fact that, while there is entanglement between the 4- and 1173-torsion point fields of $E$, there is no entanglement between the 4- and 3-torsion point fields of $E$.\footnote{Specifically, $\Q(E[4]) \cap \Q(E[1173])$ is a degree 4 extension of $\Q$, while $\Q(E[4]) \cap \Q(E[3])=\Q$.} On the other hand, it can also happen that $N=m_0$. For example, if $E=54.b2$, we see that $\rho_E$ has level 72, and this is also the level of the $m$-adic Galois representation associated to $E$.
\end{example}

\begin{algorithm}[H]
  \KwIn{$G(N)\leq \GL_2(\Z/N\Z)$ where $\im\rho_E=G$ and $N$ is the level.}
  \KwOut{$m_0 \in \Z^+$, the level of $\rho_{E,m^{\infty}}$ for $m= \prod_{\ell \in S_E} \ell$.}
  \caption{Compute Reduced level}%
  \label{alg:reduce_level}
  Let $n = \prod_{\ell \in S_E} \ell^{v_{\ell}(N)}$.\\
Compute the smallest $m_0$ dividing $n$ such that \[\#G(n)=\#G(m_0) \cdot \#\ker(\GL_2(\Z/n\Z) \rightarrow \GL_2(\Z/m_0\Z)).\]\\
  \Return{$m_0$}
\end{algorithm}

The validity of this algorithm is a consequence of the following proposition.
  \begin{proposition}\label{Prop:Level}
  Let $E/\Q$ be a non-CM elliptic curve, and let $\im \rho_E=G \leq \GL_2(\widehat{\Z})$ be a subgroup of level $N$. Define $n \coloneqq \prod_{\ell \in S_E} \ell^{v_{\ell}(N)}$ and $m\coloneqq \prod_{\ell \in S_E} \ell$. If $m_0$ is the smallest positive integer dividing $n$ such that \[\#G(n)=\#G(m_0) \cdot \#\ker(\GL_2(\Z/n\Z) \rightarrow \GL_2(\Z/m_0\Z)),\] then $m_0$ is the level of the $m$-adic Galois representation associated to $E$.
  \end{proposition}
  \begin{proof}
  First we will show that $\im \rho_{E,m^{\infty}}=\pi_1^{-1}(G(n))$, where $\pi_1\colon \prod_{\ell \in S_E} \GL_2(\Z_{\ell}) \rightarrow \GL_2(\Z/n\Z)$ is the natural reduction map. It suffices to show that $\ker \pi_1 \subseteq \im \rho_{E,m^{\infty}}$. We write $N=n\cdot n'$ with $\gcd(n,n')=1$, and so we may identify $G(N)$ as a subgroup of $\GL_2(\Z/n\Z) \times \GL_2(\Z/n'\Z)$. Under this identification, let $H$  be the intersection of $G(N)$ with the subgroup $\{I\} \times \GL_2(\Z/n'\Z)$. Then since $G$ has level $N$,
  \[
  \pi^{-1}(H) \subseteq \im \rho_E,
  \]
  where $\pi\colon \GL_2(\widehat{\Z}) \rightarrow \GL_2(\Z/N\Z)$ is the natural reduction map. The image of this subset relation under the natural projection map  $\GL_2(\widehat{\Z}) \cong \prod_{\ell} \GL_2(\Z_{\ell}) \rightarrow  \prod_{\ell \in S_E} \GL_2(\Z_{\ell}) $ gives
  \[
  \ker(\pi_1) \subseteq \im \rho_{E,m^{\infty}},
  \]
  as desired.

Let $\pi_2\colon \GL_2(\Z/n\Z) \rightarrow \GL_2(\Z/m_0\Z)$ be the natural reduction map. The assumption on $m_0$ implies that $\ker(\pi_2) \subseteq G(n)$, so it follows that $G(n) = \pi_2^{-1}(G(m_0))$. Thus if $\pi_3\colon \prod_{\ell \in S_E} \GL_2(\Z_{\ell}) \rightarrow \GL_2(\Z/m_0\Z)$ denotes the reduction map, we have
  \begin{align*}
  \pi_3^{-1}(G(m_0))&=\pi_1^{-1}(\pi_2^{-1}(G(m_0)))\\
  &=\pi_1^{-1}(G(n))\\
  &=\im \rho_{E,m^{\infty}}.
  \end{align*}
By construction $m_0$ is the smallest positive integer with this property.
  \end{proof}

\begin{corollary}
  Algorithm \ref{alg:reduce_level} is correct.
\end{corollary}

\section{Primitive Points on Modular Curves}
\label{sec:PrimitivePoints}

Let $E$ be a non-CM elliptic curve. In this section, we will reduce the question of determining whether $j(E)$ is isolated to the analysis of an associated finite set of \emph{primitive points} on modular curves. The primitive points are characterized by the following theorem.

\begin{theorem}\label{Thm:PrimPts}
Let $E/\Q$ be a non-CM elliptic curve.  There exists a finite set $\mathcal{P}=\mathcal{P}(E)$ of primitive points in $\cup_{n \in \Z^+} X_1(n)$ associated to $E$ which are characterized by the following properties:
\begin{enumerate}
\item For each $N\in \Z^+$, a point $x \in X_1(N)$ with $j(x)=j(E)$ corresponds to a unique element $x' \in \mathcal{P}$ under the natural projection map. Moreover, if $x' \in X_1(a)$, then $a \mid N$ and $\deg(x)=\deg(f)\cdot \deg(x')$, where $f\colon X_1(N) \rightarrow X_1(a)$ is the natural map.
\item The rational number $j(E)$ is isolated (respectively, sporadic) if and only if there exists an isolated (respectively, sporadic) point in $\mathcal{P}$.
\end{enumerate}
Moreover, the set $\mathcal{P}$ is minimal with respect to conditions (i) and (ii).
\end{theorem} 

\noindent Though in many ways this can be viewed as a refinement of results in \cite{BELOV}, the uniqueness of Theorem~\ref{Thm:PrimPts}\,(i) is new. In Section~\ref{sec:primitivedeg}, we give an algorithm for enumerating $\mathcal{P}(E)$ given a non-CM elliptic curve $E/\Q$.


\subsection{\texorpdfstring{Construction of $\mathcal{P}(E)$ and minimality.}{Construction of P(E) and minimality}} Let $E/\Q$ be a non-CM elliptic curve, and let $m\geq 1$ be an integer.
We begin by defining a directed graph $G(E,m)$ on the points of $X_1(n)$ corresponding to $E$ for all $n|m$. The vertices of $G(E,m)$ are tuples $(x,n,d)$ where:
\begin{enumerate}
    \item $n|m$,
    \item $x$ is a closed point on $X_1(n)$ of degree $d$, and
    \item $j(x)=j(E)$.
\end{enumerate}
We connect $(x,n,d)$ with a directed edge to $(x',n',d')$ if:
\begin{enumerate}
    \item $n'$ is a proper divisor of $n$,
    \item $x'=f(x)$ where $f\colon X_1(n)\to X_1(n')$ is the natural map, and
    \item $d=d'\cdot \deg f$.
\end{enumerate}

This is a directed acyclic graph. A \textbf{sink} in a directed acyclic graph is a vertex with no outgoing edges, and a \textbf{source} is a vertex with no incoming edges.

\begin{definition}\label{def:mprimitive}
    Let $E$ be a non-CM elliptic curve over $\Q$ and let $m\geq 1$ be an integer. The \textbf{$m$-primitive points} of $E$ are the sinks of $G(E,m)$. The \textbf{$m$-primitive degrees} are the tuples $(n,d)$, where $(x,n,d)$ is $m$-primitive for $E$. The union of all $m$-primitive points as $m$ ranges over all positive integers is the set of \textbf{primitive points} associated to $E$, denoted $\mathcal{P}(E)$. The union of all $m$-primitive degrees is the set of \textbf{primitive degrees} associated to $E$.
\end{definition}

By construction, for any $(x,n,d) \in \mathcal{P}(E)$, there does \emph{not} exist a proper divisor $n' \mid n$ with $d=\deg(f) \cdot \deg(f(x))$, where $f\colon X_1(n) \rightarrow X_1(n')$ is the natural map. Thus $\mathcal{P}(E)$ is the minimal set which can satisfy Theorem \ref{Thm:PrimPts}(i).

\begin{remark}\label{Remark:Transitivity}
Suppose $n'' | n' | n | m$ and that $(x,n,d), (x',n',d'), (x'',n'',d'')$ are vertices in $G(E,m)$. This graph is \textbf{transitive}, meaning that if there is an edge from $(x,n,d)$ to $(x',n',d')$ and from $(x',n',d')$ to $(x'',n'',d''),$ then
there is also an edge from $(x,n,d)$ to $(x'',n'',d'')$. Moreover, if there is an edge from $(x,n,d)$ to $(x'',n'',d''),$ then we  claim there is an edge from $(x,n,d)$ to $(x',n',d')$, where $x'$ is the image of $x$ on $X_1(n')$ and $n'$ is any multiple of $n''$ which properly divides $n$. Intuitively, if the degree grows as much as possible from level $n''$ to level $n$ then it also grows as much as possible from level $n'$ to level $n$. Indeed, let $f_1\colon X_1(n) \rightarrow X_1(n')$ and $f_2\colon X_1(n') \rightarrow X_1(n'')$. Suppose for the sake of contradiction that  $d<d'\cdot\deg f_1$. Since $d' \leq d'' \cdot \deg f_2$, this would imply $d<d'' \cdot\deg f_1\cdot \deg f_2$. This contradicts our assumption that there is an edge from $(x,n,d)$ to $(x'',n'',d'')$.
\end{remark}

\begin{definition}
Let $E/\Q$ be a non-CM elliptic curve and $m \in \Z^+$. For a fixed vertex $(x,n,d)$ in $G(E,m)$, consider the directed graph induced by $(x,n,d)$ and its \textbf{descendants}, i.e., all vertices $(x',n',d')$ reachable by a path from $(x,n,d)$. This is a directed acyclic graph with a single source, $(x,n,d)$.  
\end{definition}

In Corollary \ref{Cor:SingleSink} below, we will show the induced graph on the descendants of $(x,n,d)$ has a single sink as well; this is $x' \in \mathcal{P(E)}$ associated to $x$, as in Theorem \ref{Thm:PrimPts}\,(i).

\subsection{\texorpdfstring{Finiteness of $\mathcal{P}(E)$}{Finiteness of P(E)}} \label{ssec:finiteness of PE} Let $E/\Q$ be a non-CM elliptic curve, and let $m$ be the product of 2, 3, and any primes $\ell$ such that the mod $\ell$ Galois representation of $E$ is non-surjective. By Serre's Open Image Theorem \cite{serre72}, it follows that $m$ is a finite product, and there exists $m_0 \in \Z^+$ which is the level of the $m$-adic Galois representation of $E$. Suppose $v=(x,n,d) \in \mathcal{P}(E)$. Then $v$ is a sink of $G(E,N)$ for some $N \in \Z^+$. By \cite[Theorem~5.1]{BELOV}, we have 
\[
\deg(x)=\deg(f) \cdot \deg(f(x)),
\]
where $f\colon X_1(n) \rightarrow X_1(\gcd(n,m_0))$ is the natural map. 
In particular, we are using the fact that $\{2,3\} \subseteq S_E$ here.
If $\gcd(n,m_0)$ properly divides $n$, this contradicts the fact that $v$ is a sink. Thus $\gcd(n,m_0)=n$ and $n \mid m_0$. In particular, $v$ is a sink of $G(E,m_0)$. It follows that $\mathcal{P}(E)$ is the set of sinks of $G(E,m_0)$, and $\mathcal{P}(E)$ is finite since $m_0$ is a fixed positive integer depending only on $E$. We note in particular that the level of any primitive point will divide $m_0$, which in turn divides the level of the adelic Galois representation associated to $E$ (see $\S \ref{sec:madicrep}$). We record this observation in the following proposition.

\begin{proposition}\label{Rmk:Level}
  Let $E/\Q$ be a non-CM elliptic curve, and let $\mathcal{P}(E)$ denote the set of associated primitive points in $\cup_{n \in \Z^+} X_1(n)$. If $x\in X_1(a)$ is in $\mathcal{P}(E)$, then $a$ divides $m_0$.
\end{proposition}

\subsection{Preliminary Results} In this section, we establish two results concerning the residue fields of points on modular curves.

\begin{lemma}\label{lem:LCM}
Let $n_1,n_2 \in \Z^+$ and $n = \lcm(n_1,n_2)$. For $g=\gcd(n_1,n_2)$, we define $n_1'\coloneqq n_1/g$ and $n_2'\coloneqq n_2/g$.
Suppose $E/F$ is an elliptic curve, and $P \in E(\overline{F})$ is a point of order $n$. If $n_2'P\in E(F)$ and $n_1'P \in E(F)$, then $P \in E(F)$.
\end{lemma}

\begin{proof}
Note $n_2'P \in E(F)$ is a point of order $n_1$, and $n_1'P \in E(F)$ is a point of order $n_2$.  Thus there is an element $Q$ of order $n$ in
\[
\langle n_2'P, n_1' P \rangle \subseteq E(F).
\] Since $\langle n_2'P, n_1' P \rangle \subseteq \langle P \rangle$, it follows that the $F$-rational point $Q$ is a generator of $\langle P \rangle$. In particular, $P\in \langle Q \rangle$, and so $P\in E(F)$, as desired.
\end{proof}

\begin{lemma} \label{lem:compositum_res_fields}
Let $n_1,n_2 \in \Z^+$ and $n = \lcm(n_1,n_2)$. For $g=\gcd(n_1,n_2)$, we define $n_1'\coloneqq n_1/g$ and $n_2'\coloneqq n_2/g$. Let $x=[E,P] \in X_1(n)$ for an elliptic curve $E$ with $j(E) \neq 0,1728$, and define $x_1=[E,n_2'P] \in X_1(n_1)$ and $x_2=[E,n_1'P] \in X_1(n_2)$. The residue field $\Q(x)$ is at most a quadratic extension of the compositum $\Q(x_1)\Q(x_2)$. Moreover:
\begin{enumerate}
\item If $g>2$, then $\Q(x_1)\Q(x_2)=\Q(x)$.
\item If $n_1=2$ or if $n_2=2$, then $\Q(x_1)\Q(x_2)=\Q(x)$.
\end{enumerate}
\end{lemma}

\begin{proof}
Since $F_1=\Q(x_1)$ and $F_2=\Q(x_2)$ are both subfields of $F=\Q(x)$, it follows that $F_1F_2$ is as well. We will show the degree of $F/F_1F_2$ is at most 2. If $n_1=2$ or $n_2=2$, we may assume without loss of generality that $n_2=2$. Fix a Weierstrass equation for $E/\Q(j(E))$ so we may identify $F_1=\Q(j(E),\mathfrak{h}(n_2'P))$, $F_2=\Q(j(E),\mathfrak{h}(n_1'P))$, and $F=\Q(j(E),\mathfrak{h}(P))$ by Lemma \ref{ResidueFieldLemma}.

There exists $E'/F_1$ such that $\varphi \colon E \rightarrow E'$ is an isomorphism and $\varphi(n_2'P) \in E'(F_1)$; see, for example, \cite[p.~274, Proposition~VI.3.2]{DR}. Moreover, we have 
\[F_2= \Q(j(E),\mathfrak{h}(n_1'P))=\Q(j(E),\mathfrak{h}(\varphi(n_1'P)))\] by Remark \ref{ResidueFieldRmk}, and the same remark shows that the $x$-coordinate of $\varphi(n_1'P)$ is rational over $F_1F_2$. The $y$-coordinate of $\varphi(n_1'P)$ is defined over at worst a quadratic extension $L/F_1F_2$, and $L=F_1F_2$ if $n_2=2$. Then $\varphi(P) \in E'(L)$ by Lemma \ref{lem:LCM}. Since $\mathfrak{h}(\varphi(P))=\mathfrak{h}(P)$, it follows that $F \subseteq L$.

Suppose that $F/F_1F_2$ is a quadratic extension. In particular, this means $n_1 \neq 2$ and $n_2 \neq 2$. Then consider $E'/F_1F_2$, and let $\{\varphi(P),Q\}$ be a basis for $E'[n]$. Recall that if $\rho_{E',n}(\sigma)=M \in \GL_2(\Z/n\Z)$ with respect to this basis, then $M \pmod{n_1}$ gives $\rho_{E',n_1}(\sigma)$ with respect to the basis $\{n_2'\varphi(P),n_2'Q\}$. Similarly, $M\pmod{n_2}$ gives $\rho_{E',n_2}(\sigma)$ with respect to the basis $\{n_1'\varphi(P),n_1'Q\}$.  Since $n_2'(\varphi(P))$ is $F_1F_2$-rational and only the $x$-coordinate of $n_1'\varphi(P)$ is defined over $F_1F_2$, there is $\sigma\in \Gal_{F_1F_2}$ such that $\sigma({n_1'}\varphi(P))={-n_1'}\varphi(P)$ and $\sigma(n_2'(\varphi(P))=n_2'\varphi(P)$. Thus $\sigma(\varphi(P))=\alpha \varphi(P)+\beta Q$ where $\alpha \equiv 1 \pmod{n_1}$ and $\alpha \equiv -1 \pmod{n_2}$. Therefore $g$ divides $2=(\alpha+1)-(\alpha-1)$, so $g\leq 2$.
\end{proof}

\subsection{\texorpdfstring{Proof of Theorem \ref{Thm:PrimPts}\,(i)}{Proof of Theorem 24(i)}}
Theorem \ref{Thm:PrimPts}\,(i) is a consequence of a corollary to the following result.

\begin{proposition}\label{Prop:gcd}
Let $m\geq 1$ be an integer and let $E/\Q$ be a non-CM elliptic curve. In the graph $G(E,m)$,
suppose $(x,n,d)$ is connected by a path to both $(x_1,n_1,d_1)$ and $(x_2,n_2,d_2)$, where $n=\lcm(n_1,n_2)$. Then if $\gcd(n_1,n_2) \neq n_1,n_2$, it follows that both $(x_1,n_1,d_1)$ and $(x_2,n_2,d_2)$ connect to $(x_3, \gcd(n_1,n_2),d_3)$, where $x_3$ is the image of $x$ on $X_1(\gcd(n_1,n_2))$ under the natural projection map.

\end{proposition}

        \begin{figure}[h]
        \begin{center}
        \begin{tikzpicture}[node distance=1.5cm]

	\node (n){$\Q(x)$};
        \node (comp)[below of=n]{$\Q(x_1)\Q(x_2)$};
        \node (n1)[below left of=comp, node distance=2.2cm]{$\Q(x_1)$};
        \node (n2)[below right of=comp, node distance=2.2cm]{$\Q(x_2)$};
        \node (gcd)[below of=comp, node distance=3.2cm]{$\Q(x_3)$};
        \node (Q)[below of=gcd]{$\Q$};
        \draw[-] (comp) edge node [left]{$\deg(f_1)$ or $\deg(f_1)/2$}(n1);
        \draw[-] (comp) edge node [right]{$\deg(f_2)$ or $\deg(f_2)/2$}(n2);
        \draw[-] (n1) edge node [left]{$d_1/d_3$}(gcd);
        \draw[-] (n2) edge node [right]{$d_2/d_3$}(gcd);
        \draw[-] (n) edge node [left]{1 or 2}(comp);
        \draw[-] (gcd) edge node [left]{$d_3$}(Q);
        \end{tikzpicture}
        \end{center}
        \caption{Degrees of Residue Fields}
    \end{figure}
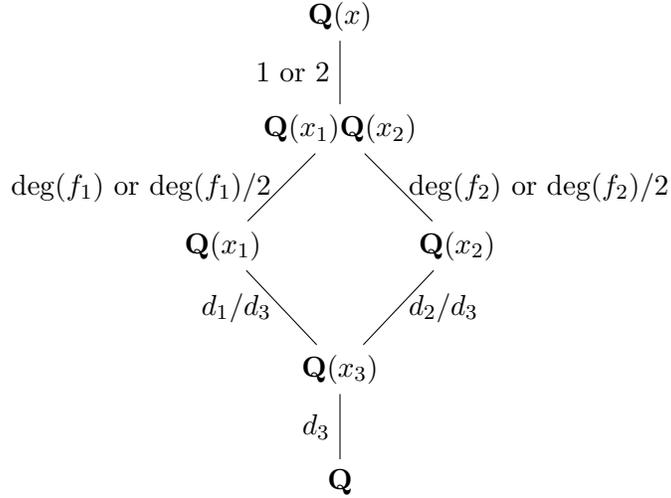

    \begin{proof}  By assumption, the integer $g=\gcd(n_1,n_2)$ is a proper divisor of both $n_1=gn_1'$ and $n_2=gn_2'$. Let $f_1 \colon X_1(n) \rightarrow X_1(n_1)$ and let $f_2 \colon X_1(n) \rightarrow X_1(n_2)$. We consider two cases.
    \begin{enumerate}
    \item Suppose $ g>2$, or if $g=1$, that there exists $n_i \leq 2$. In the latter case, our assumptions imply that exactly one of $n_1$ or $n_2$ is equal to 2 and the other is greater than 2. So if $g=1$, without loss of generality we may assume $n_1>2$ and $n_2=2$. In either case, we have $\Q(x_1)\Q(x_2)=\Q(x)$ by Lemma \ref{lem:compositum_res_fields}, and so by assumption
    \begin{align*}
    [\Q(x_1)\Q(x_2):\Q(x_1)]=\deg(f_1),\\
    [\Q(x_1)\Q(x_2):\Q(x_2)]=\deg(f_2).
    \end{align*}
        It follows from properties of composite fields that $\deg(f_1) \leq \frac{d_2}{d_3}$ and $\deg(f_2) \leq \frac{d_1}{d_3}$; see Figure 1. Note that for a prime $p \mid n_1'$, we have $p \nmid g$ if and only if $p \nmid n_2$. Similarly, for a prime $p \mid n_2'$, we have $p \nmid g$ if and only if $p \nmid n_1$. Thus by Corollary \ref{Cor:DegreeFormula}, we have
    \begin{align*}
    \frac{d_1}{d_3} \leq \deg(X_1(n_1) \rightarrow X_1(g))=\deg(f_2),\\
        \frac{d_2}{d_3} \leq \deg(X_1(n_2) \rightarrow X_1(g))=\deg(f_1).
    \end{align*}
    Thus  $d_1=d_3\cdot \deg(X_1(n_1) \rightarrow X_1(g))$ and $d_2=d_3 \cdot \deg(X_1(n_2) \rightarrow X_1(g))$, so the conclusion holds.
    \item Suppose $g=2$ or, if $g=1$, that $n_1,n_2>2$. It follows from Lemma \ref{lem:compositum_res_fields} and the same argument as above that $\frac{\deg(f_2)}{2} \leq \frac{d_1}{d_3}$ and $\frac{\deg(f_1)}{2} \leq \frac{d_2}{d_3}$; see Figure 1. However, since $n_1,n_2>2$ by assumption, we have
    \begin{align*}
    \frac{d_1}{d_3} \leq \deg(X_1(n_1) \rightarrow X_1(g))=\frac{1}{2}(\deg(f_2)),\\
       \frac{ d_2}{d_3} \leq \deg(X_1(n_2) \rightarrow X_1(g))=\frac{1}{2}(\deg(f_1)).
    \end{align*}
   Thus $d_1=d_3 \cdot \deg(X_1(n_1) \rightarrow X_1(g))$ and $d_2=d_3 \cdot \deg(X_1(n_2) \rightarrow X_1(g))$, so the conclusion follows. \qedhere
      \end{enumerate}
    \end{proof}

\begin{corollary}\label{Cor:SingleSink}
   Let $m\geq 1$ be an integer and let $E/\Q$ be a non-CM elliptic curve. For any fixed vertex $v=(x,n,d)$ in $G(E,m)$, the induced subgraph on its descendants has a single sink.
\end{corollary}

\begin{proof}
Suppose the induced subgraph has two sinks, $v_1=(x_1,n_1,d_1)$ and $v_2=(x_2,n_2,d_2)$. We will show that $n_1=n_2$, which implies $v_1=v_2$ since both $v_1$ and $v_2$ are descendants of the same vertex $(x,n,d)$. Let $g=\gcd(n_1,n_2)$. Since $v_1,v_2$ are sinks, we must have $g=n_1$ or $g=n_2$, by Proposition~\ref{Prop:gcd}, so let $g=n_1|n_2$. Remark~\ref{Remark:Transitivity} implies $n_1$ does not properly divide $n_2$ since $v_2$ is a sink. Therefore $n_1=n_2$ and $v_1=v_2$. 
\end{proof}

\subsection{\texorpdfstring{Proof of Theorem \ref{Thm:PrimPts}\,(ii)}{Proof of Theorem 24(ii)}} Let $E/\Q$ be a non-CM elliptic curve. If there exists an isolated (respectively, sporadic) point in $\mathcal{P}(E)$, then $j(E)$ is isolated (respectively, sporadic) by definition. To establish the other direction, suppose $j(E)$ is isolated. Then there exists an isolated point $x \in X_1(n)$ with $j(x)=j(E)$. By Theorem \ref{Thm:PrimPts}\,(i), we see that $x$ corresponds to a unique element $(x',n',d') \in \mathcal{P}(E)$ under the natural projection map. By the definition of $\mathcal{P}(E)$,
\[
\deg(x)=\deg(f)\cdot d',
\]
where $f\colon X_1(n) \rightarrow X_1(n')$ is the natural projection map. By \cite[Theorem 4.3]{BELOV}, the point $x'$ is isolated. The same argument shows that if $x\in X_1(n)$ is sporadic, then $x'$ is sporadic.

\section{Computing Primitive Degrees}
\label{sec:primitivedeg}

Let $E$ be a non-CM elliptic curve over $\Q$ and let $m\geq1$ be an integer. In this section, we discuss an algorithm for computing the list  of $m$-primitive degrees. At a high level, we are simply traversing the graph $G(E,m)$, always beginning at a source, until finding a sink. The sinks are the $m$-primitive points, and we record the associated $m$-primitive degree. We only retain the level $n$ and degree $d$ of a $m$-primitive point $(x,n,d)$ since, in our main algorithm, we will often try to show that $X_1(n)$ has no isolated points of degree $d$ at all. The input to the algorithm is $G=\im \rho_{E,m}\leq \GL_2(\Z/m\Z)$. We represent $m$-primitive degrees as tuples $( a,d)$ where $a$ is the level of the point $x\in X_1(a)$ and $d=\deg x$.

\begin{algorithm}[H] \caption{Compute Primitive Degrees}\label{alg:compute_prim_degree}
  \KwIn{$G\leq \GL_2(\Z/m\Z)$ such that $\im\rho_{E,m}=G$.}
  \KwOut{The multiset of $m$-primitive degrees for $E$.}
  Let $H \colonequals \langle G, -I_2\rangle \leq \GL_2(\Z/m\Z)$.\\
  Compute the orbits $O$ of $H$ acting on $(\Z/m\Z)^2$\label{Inc-1}. If $v\in (\Z/m\Z)^2 \cong E[m]$ has order $n$, then $v$ corresponds to $x\in X_1(n)$ with $j(x)=j(E)$. If $n>2$, then $\deg(x)=\#(Hv)/2$, and $\deg(x)=\#(Hv)$ otherwise.\\
   Let $D= \{\}$.\\
  For each orbit $Hv \in O$ with $v$ of order $n$, let $x \in X_1(n)$ be the associated point. Find the largest divisor $d\mid n$ such that $H(dv)$ associated to $x' \in X_1(n/d)$ satisfies $\deg(x)=\deg(x') \cdot \deg(X_1(n) \rightarrow X_1(n/d))$. Append $\langle n, (n/d, \deg(x'))\rangle$ to $D$.\label{alg:compute_primitive_degrees:minimal_levels}\\
  \Return{$D$}
\end{algorithm}

\begin{example}
Let $E/\Q$ be the non-CM elliptic curve \href{https://www.lmfdb.org/EllipticCurve/Q/147/b/1}{147.b1}, and let $m$ be the product of 2, 3, and all primes $\ell$ for which the mod $\ell$ Galois representation associated to $E$ is not surjective. Zywina's algorithm \cite{ZywinaAlgorithm} gives the image of the adelic Galois representation as the complete preimage of $G \leq \GL_2(\Z/546\Z)$, and applying Algorithm \ref{alg:reduce_level} to $G$ shows the $m$-adic Galois representation has level 78. By Proposition \ref{Rmk:Level}, the level of any primitive point in $\mathcal{P}(E)$ divides 78. Applying Algorithm \ref{alg:compute_prim_degree} to $\im \rho_{E,78}$ shows $\mathcal{P}(E)$ consists of 4 points: one point on $X_1(13)$ of degree 6, two points on $X_1(13)$ of degree 39, and one point on $X_1(1)$ of degree 1. The points on $X_1(13)$ are expected, since the mod 13 Galois representation of $E$ is not surjective. 

However, it is not necessarily the case that non-surjective primes must divide the level of some primitive point. For example, let $E/\Q$ be the non-CM elliptic curve \href{https://www.lmfdb.org/EllipticCurve/Q/232544/f/1}{232544.f1}. Then the adelic Galois representation of $E$ has level $1892$ and the $m$-adic level is 44. In particular, the mod 11 Galois representation associated to $E$ is non-surjective. However, in this case $\mathcal{P}(E)$ consists of a single point, namely, the degree one point on $X_1(1)$ associated to $E$.
\end{example}

We briefly discuss Algorithm~\ref{alg:compute_prim_degree}. Conceptually, to compute the $m$-primitive degrees of $E$, we compute for each closed point $x\in X_1(n)$ above $E$ with $n\mid m$ the unique (by Corollary~\ref{Cor:SingleSink}) $m$-primitive point induced by $x$ and record its level $a$ and degree $d$. In practice, instead of working with points on $X_1(m)$ above $E$, we compute with the matrix group $G=\im\rho_{E,m}$ and the orbits of points in $(\Z/m\Z)^2$ under the left-action of $G$. The following proposition allows us to calculate the degrees of points on modular curves from their associated orbit data.

\begin{proposition}\label{prop:compute_degrees}
Let $E/\Q$ be a non-CM elliptic curve and $m \in \Z^+$. Let $\im \rho_{E,m} \cong G \leq \GL_2(\Z/m\Z)$ and let $H=\langle G,-I\rangle$. Let $v\in (\Z/m\Z)^2$ have order $n|m$ and let $(E,P)$ be a representative of the point $x$ of $X_1(n)$  corresponding to $Gv$. If $n>2$, the degree of $x$ is $\#Hv/2$. If $n\leq 2$, the degree of $x$ is
$\#Hv=\#Gv$.
\end{proposition}

\begin{proof}
    We begin by noting that $(E, P)$ and $(E, -P)$ induce the same closed point $x$ on $X_1(n)$. Therefore, since $E$ is defined over $\Q$, the degree of
    $x$ depends only on $x(P)$ and is equal to $[\Q(x(P)):\Q]$; see Remark \ref{ResidueFieldRmk}.
    We next observe that $[\Q(P):\Q]=\#Gv$. Assume that $n>2$. When $-I\in G$, the points $P$ and $-P$ are distinct and in the same Galois orbit so $[\Q(x(P)):\Q]=\frac{1}{2}[\Q(P):\Q]$ and $\#Hv=\#Gv$. We conclude that
    \[
    \deg x = [\Q(x(P)):\Q] = \frac{1}{2}[\Q(P):\Q] = \frac{1}{2}\#Gv = \frac{1}{2}\#Hv.
    \]
    If $-I \notin G$, then there exists a twist $E'$ of $E/\Q$ such that $\im \rho_{E',n}=\langle G, -I \rangle$ by \cite[Corollary 5.25]{sutherland}. The point $(E,P)$ is also represented by $(E',P')$ for some $P'\in E'[n]$, and the same argument from above  implies
    \[
    \#Hv = [\Q(P'):\Q] = 2[\Q(x(P')):\Q] = 2[\Q(x(P)):\Q],
    \]
    so we again have that the degree of the point represented by $(E,P)$ is $\frac{1}{2}\#Hv$.

    Now assume $n\leq 2$. Then $-I=I$ in $\GL_2(\Z/n\Z)$ so $G=H$, and $\Q(x(P))=\Q(P)$. We conclude
    \[
    \deg x = [\Q(x(P)):\Q] = [\Q(P):\Q] = \#Gv = \#Hv. \qedhere
    \]
\end{proof}

\section{Genus 0 adelic images do not produce isolated points}
\label{sec:genus0images}

\label{sec:genus0}
Let $E/\Q$ be an elliptic curve and $G\leq \GL_2(\hat \Z)$ its adelic image. Let $G(N)$ denote the image of its mod $N$ representation. Denote by $B_1(N)$ the subgroup of $\GL_2(\Z/N\Z)$ consisting of the upper triangular matrices with a $1$ in the upper left entry. Note that $X_{B_1(N)}=X_1(N)$. We say that a congruence group $\Gamma$ is of genus $g$ if $X_\Gamma$ is of genus $g$. We say that a point $x$ corresponds to an elliptic curve $E$ if $j(x) = j(E)$. In this section, we show that elliptic curves with genus $0$ mod $N$ image  do not correspond to $\PP^1$-isolated points on $X_1(N)$.

\begin{lemma} \label{lem:map to Symd}
Let $f\colon X \to Y$ be a finite morphism of curves of degree $d$. Then $f$ induces a non-constant morphism $f^*\colon Y \to X^{(d)}$.
\end{lemma}
\begin{proof}
    The point is to show that the natural map $y \mapsto f^{-1}(y)$ is a morphism of schemes.  Note that the composition of $\Gamma_f\colon X \to X \times Y$ sending $X$ to the graph of $f$ with $X \times Y \to Y$ is just $f$, which is flat as a finite morphism of irreducible curves~\cite[Proposition~4.3.9]{Liu2002}. So the graph of $f$ defines a relative effective Cartier divisor in the sense of~\cite[Definition~3.4]{MilneJV} on $X \times Y/Y$ of degree $d$.
    Since all our schemes are regular, we can identify Cartier with Weil divisors. 
    In the hypothesis of the statement of~\cite[Theorem~3.13]{MilneJV}, we can thus take the effective divisor to be the graph of $f$. This allows us to conclude that, as a map of sets, $f^*$ maps $y \in Y$ to the degree-$d$ divisor $f^{-1}(y) \colonequals [\{y\} \times_Y X]$ with the multiplicities of the reduced subscheme of the points in the fiber product equal to the ramification indices (see~\cite[\S\,1.5, \S\,1.7]{FultonIntersectionTheory}).
     
    The morphism $f^*$ is non-constant since fibers above different points are mapped to different points by $f^*$.
\end{proof}


The following lemma rephrases the definition of a point being $\PP^1$-parametrized of degree $d$. 
\begin{lemma}[Characterization of $\bP^1$-parametrized points] \label{lem:iso1}
Let $X/k$ be a curve. Let $x \in X^{(d)}(k)$ be an irreducible degree $d$ divisor. The following are equivalent:
\begin{enumerate}
    \item The point $x$ is $\bP^1$-parametrized.

    \item There is a non-constant morphism $\bP^1 \rightarrow X^{(d)}$ containing $x$ in its image.
\end{enumerate}
\end{lemma}
\begin{proof}
(i)$\implies$(ii): If $x$ is $\PP^1$-parametrized, then there is an $x' \in X^{(d)}(k)$ different from $x$ such that 
$x-x'=\mathrm{div}(f)$ for a non-constant function $f\in k(X)$. Here we treat $x$ and $x'$ as effective divisors of degree $d$ on $X$. This gives a degree $d$ map $f\colon X \to \PP^1$, which gives the copy of $\PP^1$ inside $X^{(d)}$ as the image under pullback of $f$ as in Lemma~\ref{lem:map to Symd}.  

(ii)$\implies$(i): 
Since every rational map from a unirational variety to an abelian variety is constant~\cite[Corollary~3.8]{MilneAV}, the $\PP^1$ is contracted to a point under $\Phi_d\colon X^{(d)} \rightarrow \Jac(X)$.
\end{proof}

\begin{lemma} \label{lem:iso2}
Let $f\colon X\rightarrow Y$ be a finite morphism of curves and $x$ a closed point on $X$, and assume  $\deg x = \deg f(x)$. If $x$ is $\PP^1$-parametrized, then so is $f(x)$.
\end{lemma}
\begin{proof}
Let $d=\deg(x)$. The proof is diagram chasing using $\PP^1 \to X^{(d)}\stackrel{f^{(d)}}{\to} Y^{(d)}$, where $f^{(d)}$ is the natural map from $X^{(d)}$ to $Y^{(d)}$ induced by $f$.
If $x\in X$ is not $\bP^1$-isolated (i.e., $x$ is $\bP^1$-parametrized), then
by Lemma~\ref{lem:iso1}, $x$ can be viewed as a point on $X^{(d)}$ lying on the image of a non-constant map from $\PP^1$ to $X^{(d)}$.  
Furthermore, $f(x)$ lies on the image of this $\PP^1$ (which is again a $\PP^1$ by \Cref{lem:map to Symd} or because the induced morphism on the $d$-th symmetric power is again finite) inside of $Y^{(d)}$. Now the conclusion follows by \Cref{lem:iso1}.
\end{proof}

The following lemma was communicated to us by Maarten Derickx. By $X_K$ we denote the base change of a curve $X$ to a field $K$.
\begin{lemma*}[Descent lemma (Derickx)] Let $X/\Q$ be a curve, $K$ a number field of degree $[K:\Q]=d_2$, and let $x_0\in X_K$ be a closed point (over $K$) of degree $d_1$ that is $\PP^1_K$-parametrized. Let $x$ be the image of $x_0$ under the map $X_K\rightarrow X$. If $K \subseteq \Q(x)$, then $x$ is $\PP^1_\Q$-parametrized of degree $d=d_1d_2$.

\end{lemma*}
\begin{proof}
    After composing with an appropriate automorphism of $\PP^1_K$ if necessary, we let $f:\PP^1_K\rightarrow X_K^{(d_1)}$ be such that $f(\infty)=x_0$. Define the map $g:\PP^1_\Q\rightarrow X_\Q^{(d)}$ as $f^{\sigma_1}+ \ldots +f^{\sigma_{d_2}}$, where the $\sigma_i$ are the embeddings of $K$ into $\overline K$; it is a map defined over $\Q$ by descent theory. As 
    $$x=x_0^{\sigma_1} + \ldots +x_0^{\sigma_{d_2}}=f(\infty)^{\sigma_1}+ \ldots + f(\infty)^{\sigma_{d_2}}=g(\infty),$$
    it follows that $x$ lies in $g(\PP^1_\Q(\Q))$, so it is $\PP^1_\Q$-parametrized by Lemma \ref{lem:iso1}.
\end{proof}

\begin{theorem} \label{Genus0Prop}
Let $E/\Q$ be an elliptic curve with mod $N$ image $G(N)$ of genus $0$ and $N$ a positive integer. 
Then every $x \in X_1(N)$ with $j(x) = j(E)$ is $\PP^1$-parametrized.
\end{theorem}
\begin{proof} 
Let $x=[E,P]\in X_1(N)$. Replacing $G(N)$ with an appropriate choice of conjugation if necessary, we may assume $G(N)$ is with respect to a basis having $P$ as its first element. By replacing $E/\Q$ with a twist if necessary, we may assume $-I \in G(N)$; see \cite[Corollary 5.25]{sutherland}. Let $B \colonequals B_1(N) \cap G(N)$ in $G(N)$, and let $S:=\det B \leq (\Z/N\Z)^\times$. The group $S$ is canonically isomorphic to a subgroup of $\Gal(\Q(\zeta_N)/\Q)\simeq (\Z/N\Z)^\times$, with which we identify it. Let $K$ be the fixed field of $S$ and $d_2=[K:\Q]=(\Z/N\Z)^\times /\#S$. Note $K \subseteq \Q(x)$ by construction. Let $y$ be a $K$-rational point on $X_{G(N),K}$ corresponding to $E/K$, with respect to our chosen basis.

 Let $\pi:\GL_2(\Z) \rightarrow \GL_2(\Z/N\Z)$ be reduction modulo $N$, and $\Gamma(B)$ be $\pi^{-1}(B)\cap \SL_2(\Z)$. Similarly, we let $\Gamma(G(N)) \coloneqq \pi^{-1}(G(N))\cap \SL_2(\Z)$. We define $X'_B$ to be the modular curve $\Gamma(B) \backslash \mathcal H ^*$, viewed as a geometrically integral curve over $K$. We note that this definition differs from the definition of $X_B$  as in \Cref{sec:modular_curves} if $B$ has non-surjective determinant: $X_B$ is defined over $\Q$ and is geometrically reducible.

 The proof is diagram chasing (over $K$) in the following two diagrams while keeping track of the degree of the point.  The situation is as follows:
\[
\begin{tikzcd}
     &X'_B \arrow[rd, "f"] \arrow[dl,"g",swap] &  \\
     X_1(N)_K & &X_{G(N),K} \isom \PP^1_K
\end{tikzcd}
     \Longrightarrow
\begin{tikzcd}
      &X'_B{}^{(d_1)} \arrow[dl,"g^{(d_1)}",swap] & \\
     X_1(N)^{(d_1)}_K & &X_{G(N),K} \isom \PP^1_K\arrow[ul, "f^*",swap]
\end{tikzcd}
\]
Let $f\colon X'_{B}\rightarrow X_{G(N),K} \isom \PP^1$ be the corresponding map of modular curves. By \cite[p.66]{modular} we have $\deg(f)=[\pm \Gamma(G(N)):\pm\Gamma(B)]=d_1$.
\Cref{lem:map to Symd} yields a non-constant morphism $f^*\colon X_{G(N),K} \to X'_{B}{}^{(d_1)}$ such that $f^*(y)$ lies on a $\PP^1_K \isom X_{G(N),K}$. 
Hence the irreducible divisor represented by $f^*(y)$ in $X'_B$ is $\PP^1_K$-parametrized by~\Cref{lem:iso1}. Consider the morphism $g\colon X'_B \to X_1(N)_K$ corresponding to $B \leq B_1(N)$. Let $x'\in X'_B$ be such that $f(x')=y$. Then $g(x')=x_0$ is a closed point on $X_1(N)_K$ (over $K$) which maps to $x$ under the map $X_1(N)_K \rightarrow X_1(N)_{\Q}$. That is, $x_0$ is an irreducible component of the base change of $x$ to $K$. We want to show that $g(x')$ has the same degree as $x'$ (all as closed points over $K$). Note that $\deg(x')$ is $d_1$ and $\deg(x_0)=[G(N)_K: \pm B]$, where $G(N)_K$ denotes the mod $N$ image of $E/K$, as $\pm B$ is the stabilizer of $x_0$. Hence it satisfies $\deg(g(x')) = \deg(x')$, so $x_0$ is $\PP^1_K$-parametrized by~\Cref{lem:iso2}.
Now it follows that $x$ is $\PP^1_\Q$-parametrized by the Descent lemma. 
\end{proof}

\begin{corollary}\label{cor:adelicgaloisgenus0}
Let $E/\Q$ be an elliptic curve with adelic image $G$ of genus $0$ and $n$ a positive integer. 
Then 
every $x \in X_1(n)$ with $j(x) = j(E)$ is $\PP^1$-parametrized.
\end{corollary}
\begin{proof}
Since $G$ is by assumption of genus $0$, it follows that so is $G(n)$. The result then follows from Theorem~\ref{Genus0Prop}.
\end{proof}

\begin{example}\label{eg:modngenus0}
  Consider the elliptic curve $E$ with LMFDB label \href{https://www.lmfdb.org/EllipticCurve/Q/15/a/7}{15.a7}, which has adelic image of genus 13. We can show $E$ has the following primitive points:
\begin{itemize}
    \item $X_1(1)$ of degree 1,
    \item $X_1(2)$ of degrees 1 and 2,
    \item $X_1(4)$ of degree 1,
    \item $X_1(8)$ of degree 2,
    \item $X_1(16)$ of degree 4, and
    \item $X_1(32)$ of degree 8.
\end{itemize}
By Theorem \ref{Thm:PrimPts}, it follows that $j(E)=-1/15$ is isolated if and only if one of these points is isolated. Any point $x \in X_1(n)$ with $\deg(x)>\text{genus}(X_1(n))$ has Riemann--Roch space of dimension at least 2 and thus is $\PP^1$-parametrized. It remains only to address the point on $X_1(32)$ of degree 8. However, the mod $32$ Galois representation of $E$ has genus $0$, so this point is $\PP^1$-parametrized by Theorem \ref{Genus0Prop}. We may conclude that $j(E)=-1/15$ is not an isolated $j$-invariant.
\end{example}

\section{Validity of Main Algorithm}
\label{sec:validity}

In this section, we will prove that our Main Algorithm as stated in Section~\ref{sec:mainalg} is valid. The section concludes with an additional example.

\begin{theorem}\label{thm:validity_main_algorithm}
Let $j\in\Q$ be a non-CM $j$-invariant. If Algorithm \ref{alg:mainalgorithm} returns $\{( a_1,d_1), \dots, ( a_k,d_k) \}$, then any isolated (respectively, sporadic) point $x\in X_1(N)$ for $N \in \Z^+$ with $j(x)=j$ maps under the natural projection map to an isolated (respectively, sporadic) point of degree $d_i$ on $X_1(a_i)$ for some  $1 \leq i \leq k$.
\end{theorem}

\begin{remark}
Since our algorithm builds on that of Zywina \cite{ZywinaAlgorithm}, it is possible that our algorithm will give an error if the adelic image cannot be computed. See Zywina \cite{ZywinaImagesGit} for details. In particular, he notes that ``errors will always occur if $E$ gives rise to an unknown exceptional rational point on certain high genus modular curves." There are no known examples where this error occurs, and in particular it does not occur for any elliptic curves in the LMFDB or Stein--Watkins Database. An algorithm to compute the adelic image does always exist (see work of Brau Avila \cite{Avilo15}), though it is not practical in general.
\end{remark}

\begin{remark}
The range of the moduli to which \texttt{PrimitiveDegreesOfPoints} is applicable is restricted by the amount of memory Magma can use. The input of very large matrix groups may result in a runtime error. Because we take preliminary steps to reduce the modulus of the matrix group (i.e., Section \ref{sec:madicrep}), this error did not occur when running our full algorithm on all elliptic curves currently in the LMFDB.
\end{remark}

\begin{proof}
Let $E/\Q$ be a non-CM elliptic curve with $j(E)=j$. We note that the choice of $E$ will not impact our result; see Section \ref{sec:ClosedPts}. We may compute $\im \rho_E=G$ via Zywina's algorithm \cite{ZywinaAlgorithm}, and represent the output as $G(N) \leq \GL_2(\Z/N\Z)$ where $N$ is the level. By Algorithm \ref{alg:reduce_level}, we may use $G(N)$ to compute the level $m_0$ of the $m$-adic Galois representation associated to $E$, where $m$ is the product of 2, 3, and all non-surjective primes. By Proposition \ref{Rmk:Level}, the level of any primitive point in $\mathcal{P}(E)$ will divide $m_0$, so applying Algorithm \ref{alg:compute_prim_degree} to $\im \rho_{E,m_0}$ results in the complete set of primitive degrees for $E$. In particular, for each closed point $x \in X_1(n)$ with $j(x)=j$ and $n \mid m_0$, we have $x$ mapping to the primitive point $x' \in X_1(a)$ of degree $d$. We record the entry $\langle n, ( a, d) \rangle$ in the multiset $D$. By Theorem \ref{Thm:PrimPts}, we have $\deg(x)=\deg(x') \cdot \deg(X_1(n) \rightarrow X_1(a))$, and $j$ is isolated (respectively, sporadic) if and only if there is an isolated (respectively, sporadic) point on $X_1(a_i)$ of degree $d_i$ for some $\langle n_i, ( a_i, d_i) \rangle$ in $D$.


Next, we will rule out pairs $(a,d)$ which cannot correspond to isolated points. If $d> \text{genus}(X_1(a))$, then the point is not $\mathbf{P}^1$-isolated since its associated Riemann--Roch space has dimension at least 2. Thus we need only consider the multiset $D' \subseteq D$ containing those elements $\langle n,( a, d) \rangle$ for which $d \leq \text{genus}(X_1(a))$. The integer $n$ is not relevant for our purposes, so we create a new multiset $M$ containing $(a,d)$ from $D'$. We record $(a,d)$ with multiplicity $\mu$ if and only if $X_1(a)$ has $\mu$ distinct closed points of degree $d$ which are associated to $E$.

Finally, by Theorem \ref{Genus0Prop}, we may remove from $M$ any pair $(a,d)$ where $\im \rho_{E,a}$ corresponds to a modular curve of genus 0. Since Algorithm \ref{alg:mainalgorithm} returns $M$, we are done.
\end{proof}

\begin{corollary}\label{cor:P1_isolated}
If Algorithm \ref{alg:mainalgorithm} outputs an empty set on some non-CM $j$-invariant $j \in \Q$, then $j$ is not the image of an  isolated point on $X_1(N)$ for any positive integer $N$.
\end{corollary}

\begin{example}\label{Ex45}
We end this section with an extended example to illustrate the impact of each step in Algorithm \ref{alg:mainalgorithm}. Let $E=1225.b1$, a non-CM elliptic curve over $\Q$ with $j(E)=-162677523113838677$, and let $m$ denote the product of 2, 3, and all non-surjective primes. Here, one can check that $m=162$.
\begin{enumerate}
\item Zywina's algorithm shows the adelic image of $E$ is of level $N=5180$. Let $G(N) \coloneqq \im \rho_{E,N}$.
\item Algorithm \ref{alg:reduce_level} shows that the level of the $m$-adic Galois representation of $E$ is 148.
\item Algorithm \ref{alg:compute_prim_degree} shows that it suffices to consider points on $X_1(n)$ where $n \mid 37$. In particular, it is not necessary to consider points on $X_1(74)$ or $X_1(148)$ since any will map under the natural projection map to a modular curve of lower level. There is a single point of degree 1 on $X_1(1)$ and 4 points on $X_1(37)$ --- one of degree 18 and 3 of degree 222.
\item We compute that $\text{genus}(X_1(37))=40$ and  $\text{genus}(X_1(1))=0$. Thus no point of degree 222 on $X_1(37)$ can be isolated, and neither is the rational point on $X_1(1)$.
\item Since the modular curve associated to $\im \rho_{E,37}$ has genus 4, the algorithm returns $ \{( 37, 18)\}$.
\end{enumerate}
Since $X_1(37)$ has infinitely many degree 18 points by \cite{DerickxVanHoeij2014}, the point of degree 18 on $X_1(37)$ is not sporadic. It follows that $j(E)$ is not a sporadic $j$-invariant. In fact, the appendix shows that $j(E)$ is isolated; see Theorem \ref{DisolatedThm}.
\end{example} 

\section{Remaining Filters and Computational Results}
\label{sec:remainingfilters}

Suppose $E/\Q$ is non-CM elliptic curve such that one of the following holds:
\begin{itemize}
    \item$N_E \leq \numprint{500000} $,
    \item $N_E$ is only divisible by primes $p \leq 7$, or
    \item $N_E=p \leq \numprint{300000000}$ for some prime number $p$.
\end{itemize} Then running Algorithm \ref{alg:mainalgorithm} on $E$ results in the empty set, aside from the $j$-invariants listed in Table~\ref{table:imagegt0}; the output from elliptic curves in the Stein--Watkins database yields no additional $j$-invariants. The final step in the main algorithm filters out curves with mod $N$ genus $0$; we also list in the table the mod $N$ genus.  This genus is computed using code associated to the paper \cite{RSZB2022}. 
	\begin{table}[h]
	\begin{center}
		\begin{tabularx}{266pt}{l l c} \toprule
		$j$ & $( N, d )$ & genus mod $N$  \\\midrule
$-140625/8$ &  $\{ ( 21, 3 )^2 \} $ & $1$\\
 $-162677523113838677$& $\{ ( 37, 18 ) \}$ & $4$\\
 $-882216989/131072$ & $\{ ( 17, 4 )^2\}$ & $1$ \\
 $-9317$ & $\{ ( 37, 6 )^3 \} $ & $4$\\
 $16778985534208729/81000$ & $\{ ( 24, 4 )^2 \}$ & $1$\\
 $351/4$ & $\{ ( 28, 9 )^2 \}$ & $5$\\ \bottomrule
		\end{tabularx}
		\caption{Output of main algorithm}\label{table:imagegt0}
	\end{center}
\end{table}

In this section we describe additional computations that prove that the only $\PP^1$-isolated points on $X_1(N)$ for a fixed $N$ correspond to the four $j$-invariants $j=-140625/8,-9317,351/4,$ and $ -16267752311383867$, proving Theorem \ref{LMFDBoutputThm}. Since $j=-140625/8,-9317$, and $351/4$ are known to be isolated (see Section \ref{sec:intro}), it suffices to consider only the remaining 3 $j$-invariants.
\if{false}
\begin{proposition} \abbey{New plan is to use result about genus 0 mod $N$ images in place of this Proposition (instead of working to correct it).}
    Let $E/\Q$ be an elliptic curve and let $G_0$ be the image mod $N$ Galois representation attached to $E$, viewed as a subgroup of $\GL_2(\Z/N\Z)$ after some choice of basis for $E[N]$. Let $y=(E,P)$ be a degree $d$ point on $X_1(N)$. Suppose there is a conjugate $c(B_1(N))$ of $B_1(N)$ such that $B:=\langle c(B_1(N)), G_0 \rangle \leq \GL_2(\Z/N\Z)$ and $X_B$ is of genus 0. If $[\pm B: \pm c(B_1(N))]=d$, then $y$ is not $\PP^1$-isolated.
\end{proposition}

\abbey{We need some sort of compatibility criterion here. For example, suppose $G_0$ is the full Borel subgroup, so $X_B=X_0(N)$. Then for $y$ on $X_1(N)$ of degree $d$, it does not necessarily follow that $f(y)$ is in $X_0(N)(\Q).$ We would need to know that $\langle P \rangle$ is actually $\Q$-rational. I think we need to choose a conjugate of $G_0$ so that $P$ is the first basis element, and then work with $B_1(N)$ and not $c(B_1(N))$.}

\begin{proof}
    By \Cref{prop:mapdegree} it follows that there exists a degree $d$ map $f\colon X_1(N)\rightarrow X_B\simeq \PP^1$ of degree $d$. Since $G_0$ is contained in $B$, it follows that $E$ corresponds to  $x\in X_B(\Q)$ such that $f(y)=x$ \abbey{I don't think we necessarily have $f(y)=x$; see comment directly above proof.}. Hence $y$ is not $\PP^1$-isolated.
\end{proof}

\fi

\subsubsection{Degree $18$ point on $X_1(37)$ corresponding to $j=-162677523113838677$}  \label{18_37}

We show that the point $x\in X_1(37)(K)$, where $K$ is a degree 18 number field, and $j(x)=-162677523113838677$ is $\PP^1$-isolated. We explicitly compute the coordinates of $x$ on a model of $X_1(37)$ and define  $\sigma_i$, for $i=1,\ldots, 18$ to be the automorphisms of $K$ and $D=\sum_{i=1}^{18}\sigma_i(x)$. Reducing everything modulo $3$ and denoting the reduction of $D$ modulo $3$ by $\overline{D}$, we obtain $\ell(\overline{D})=1$, which shows that the reduction
$\overline{x}$ of $x$ modulo 3 is $\PP^1$-isolated. Hence it follows that $x$ is $\PP^1$-isolated.  

\subsubsection{Degree $4$ point on $X_1(17)$ corresponding to $j=-882216989/131072$ and degree $4$ point on $X_1(24)$ corresponding to $j=16778985534208729/81000$ }
To show that these points are not isolated we compute the coordinates of a degree $4$ point $x$ corresponding to our curve on a model of $X_1(17)$ and $X_1(24)$. 
Let $\sigma_i$, for $i=1,\ldots, 4$ be the automorphisms of $K$ and $D=\sum_{i=1}^4\sigma_i(x)$. In each case, we compute $\ell(D)=2$, which implies that $x$ is not $\PP^1$-isolated.

The fact that the degree $4$ point on $X_1(17)$ is not $\PP^1$-isolated also follows from the results of \cite[Proposition 6.7]{DerickxMazurKamienny}, where it is shown that there are no $\PP^1$-isolated quartic points on $X_1(17)$.

\subsection{Computations}
We give some details on the implementation and runtime of the algorithm in this section.
We ran Algorithm \ref{alg:mainalgorithm}
on two databases of $j$-invariants of elliptic curves over $\Q$:
\begin{itemize}
\item The LMFDB \cite{LMFDB} contains all elliptic curves over $\Q$ of conductor up to \numprint{500000} and includes roughly $2$~million $j$-invariants.
\item The Stein--Watkins database \cite{SteinWatkins} contains roughly $36$~million unique $j$-invariants of elliptic curves over $\Q$ of absolute discriminant at most $10^{12}$ that have conductor at most $10^8$ or prime conductor at most $10^{10}$. However, it has been filtered to include just one representative from each isogeny class and each class of quadratic twists.
\end{itemize}
All computations were run on a server with an AMD EPYC 7713 2GHz CPU and Magma V2.28-3 \cite{Magma}.
For every elliptic curve $E/\Q$ in the LMFDB, the database contains the genus of $X_G$ where $G$ is the image of the adelic Galois represenation of $E$. Applying Corollary \ref{cor:adelicgaloisgenus0}, we filtered the roughly $2$~million $j$-invariants of non-CM elliptic curves $E/\Q$ in the LMFDB to a set of \numprint{30141}, such that the image of the adelic Galois representation of $E$ is greater than $0$. On this set, the computation took $2$~CPU hours and $1714$ MB of memory. The Stein--Watkins database contains \numprint{35788699} unique $j$-invariants of non-CM elliptic curves over $\Q$, with no information about the image of the adelic Galois representation. Running Algorithm \ref{alg:mainalgorithm} on this database took $442$ CPU hours.

\section{\texorpdfstring{Appendix: {The $j=-162677523113838677$ example is AV-isolated}\\ by Maarten Derickx and Mark van Hoeij}{Appendix: {The j=-162677523113838677 example is AV-isolated} by Maarten Derickx and Mark van Hoeij}}


Let $X = X_1(37)$ and $D$ be the divisor for the degree 18 closed point in Section~\ref{18_37}, which also appeared in \cite[Remark 1.3]{Ejder}.
Through the maps $X \rightarrow X_0(37) \rightarrow X_0^+(37)$, we can view the Jacobian of the latter, $A \colonequals J_0(37)^+$,
as an abelian subvariety of $J_1(37)$.
In the isogeny decomposition of $J_1(37)$ there is exactly one abelian variety of positive rank, namely $A$. This means that the question of whether $D$ is 
AV isolated is equivalent to $\Phi_d(D) + A$ not being contained in $W_{18}(X)$,
where $W_d(X) \colonequals \Phi_d(X^{(d)})$ with $\Phi_d$ as in Section~\ref{sec:2.1}.

\begin{lemma}
Let $A$ be an abelian subvariety of ${\rm Jac}(X)$ and let
$\pi\colon  X^{(d)} \rightarrow J(X)/A$ be the composition of $\Phi_d$ with the quotient map.
Let $D \in  X^{(d)}$ and suppose that $\Phi_d(D) + A \subseteq W_d(X)$.
Then $\pi$, restricted to the tangent space at $D$, is not injective.
\end{lemma}

\begin{proof} Assume that $\Phi_d$ is injective on the tangent space at $D$,
otherwise there is nothing to prove as $\pi$ factors through $\Phi_d$.
Then $\Phi_d$ induces an isomorphism between the tangent spaces of $X^{(d)}$ at $D$, and $W_d$ at $\Phi_d(D)$.
The positive dimensional variety $\Phi_d(D) + A$ is contracted to a point under the quotient map
$J(X) \rightarrow J(X)/A$, in particular the entire tangent space of  $\Phi_d(D) + A$ at $\Phi_d(D)$
is sent to zero under the quotient map, and the result follows.
\end{proof}

\begin{corollary} If the corresponding map $\pi^* \colon {\rm Cot}_0 \,J(X)/A \rightarrow {\rm Cot}_D \,X^{(d)}$
on cotangent spaces is surjective, then there is no translate of $A$ contained in $W_d(X)$ passing through $\Phi_d(D)$.
\end{corollary}

\begin{proof} This is because a map on tangent spaces is injective if and only if the corresponding map
on cotangent spaces is surjective.
\end{proof}
\begin{theorem} \label{DisolatedThm}
    $D$ is AV-isolated.
\end{theorem}
\begin{proof}
The corollary says that if $\pi$ is a {\em formal immersion} \cite{Version1}
        %
        %
        %
at $D$ then $D$ is AV-isolated.
To prove that $\pi$ is a formal immersion at $D$, we can use~\cite[Proposition~3.7]{Version1}
if we have a basis of the cotangent space of $A$ viewed as a subspace of $J(X)$.

Write $D = y_1 + \cdots + y_{18}$ where the $y_i$ form the Galois orbit corresponding
to the degree 18 closed point.
These $y_i$ also form one orbit under the diamond operators (they map to the same point in $X_0(37)$).
We order them in such a way that $y_i = \langle 2^i\rangle y_1$.
Let $q_1$ be a uniformizer at $y_1$, and let $q_i \colonequals q_1 \circ \langle 2^{-i} \rangle$
be the corresponding uniformizer at $y_i$. Note that since $A$ is inside $J_0(37)$
we have $(\langle 2\rangle - 1)(A) = 0$, so $\langle2\rangle  - 1$ factors via $J(X)/A$.
In particular all one-forms of the form $\langle 2^i\rangle (\langle 2\rangle  - 1) w$ inside ${\rm Cot}_0 \,J(X)$ come from ${\rm Cot}_0 \,J(X)/A$.

Next we show that there is a one-form $w$ such that $a(w,q_1,1) = 1$ and $a(w,q_i,1)=0$ for all $i > 1$, with $a(\ldots)$ as
in~\cite[Proposition~3.7]{Version1}.
By Riemann--Roch, \[{\rm dim\,} H^0(O_X(y_1+ \ldots +y_{18})) - {\rm dim\,} H^0(\Omega^1(-y_1-\ldots-y_{18})) = 18 + 1 - 40\]
and
\[{\rm dim\,} H^0(O_X(y_2+ \ldots +y_{18})) - {\rm dim\,} H^0(\Omega^1(-y_2-\ldots-y_{18})) = 17 + 1 - 40.\] 
Now ${\rm dim\,} H^0(O_X(y_1 + \ldots +y_{18})) = 1$ (equivalent to the degree 18 point being $\PP^1$-isolated)
and thus is equal to ${\rm dim\,} H^0(O_X(y_2+\ldots+y_{18}))$. 
Then by Riemann--Roch, \[
{\rm dim\,} H^0(\Omega^1(-y_2-\ldots-y_{18})) > {\rm dim\,} H^0(\Omega^1(-y_1-\ldots-y_{18})).\]
So the wanted one-form $w$ can be found by picking an element of $H^0(\Omega^1(-y_2-\ldots-y_{18}))$
that is not in $H^0(\Omega^1(-y_1-\ldots-y_{18}))$ and scaling it to make  $a(w,q_1,1) = 1$. 

Now let $w_i =  \langle 2^i\rangle(\langle2\rangle - 1)w$. These one-forms all come from ${\rm Cot}_0 \,J(X)/A$ as mentioned before.
The matrix in~\cite[Proposition~3.7]{Version1} now has the following form:

\[
\begin{pmatrix}
1& -1& 0& 0& \ldots& 0 \\
0& 1& -1& 0& \ldots& 0 \\
0& 0& 1& -1& \ldots & 0 \\
\vdots&  \vdots  & \vdots   &   \ddots &  \ddots & \vdots\\
0& 0& 0& \ldots & 1 & -1
\end{pmatrix}
\]

This matrix has rank 17, but for 
\cite[Proposition~3.7]{Version1}
we need rank 18.
An additional one-form $w'$ with $a(w',q_1,1) = \ldots = a(w',q_{18},1) = 1$ would increase
the rank to 18.
We can take $w'$ to be the pullback of a one-form on $J_0(37)^-$ to $J(X)$. Indeed, $J_0(37)^-$ is a quotient of
$J_0(37)/A$ and thus of $J(X)/A$. Such a $w'$ is invariant under the diamond operators so $a(w',q_1,1) = \ldots = a(w',q_{18},1)$
automatically holds. It remains to show that it does not vanish at $y_1$.  
Since $J_0(37)^-$ is an elliptic curve, $w'$ does not vanish anywhere,
and in particular not at the image of $y_1$ in $J_0(37)^-$.
The only possibility for $w'$ to vanish at $y_1$ is then for $X_1(37) \rightarrow J_0(37)^-$
to be ramified at $y_1$, however, this is not the case since the ramification points correspond to elliptic curves whose $j$-invariant is not equal to $-162677523113838677$.
Indeed, $X_1(37) \rightarrow X_0(37)$ only ramifies at $j=0$ and $j=1728$, and a Magma computation shows that both ramification points of $X_0(37) \rightarrow J_0(37)^-$ have $j$-invariant $287496$.
\end{proof}

\newcommand{\etalchar}[1]{$^{#1}$}

\end{document}

%% file: author_info.tex
\author[Bourdon]{Abbey Bourdon}
\address{
  Abbey Bourdon,
  Wake Forest University,
  Department of Mathematics, 127 Manchester Hall, PO Box 7388, Winston-Salem, NC 27109
}
\email{\url{bourdoam@wfu.edu}}

\author[Hashimoto]{Sachi Hashimoto}
\address{  Sachi Hashimoto, Department of Mathematics,
  Brown University,
  Box 1917,
  151 Thayer Street,
  Providence, RI, 02912
}
\email{\url{sachi_hashimoto@brown.edu}}
\urladdr{\url{https://sachihashimoto.github.io/}}

\author[Keller]{Timo Keller}
\address{Timo Keller, Rijksuniveriteit Groningen, Bernoulli Institute, Bernoulliborg, Nijenborgh 9, 9747 AG Groningen, The Netherlands (previously: Leibniz Universität Hannover, Institut für Algebra, Zahlentheorie und Diskrete Mathematik, Welfengarten 1, 30167 Hannover, Germany)}
\email{t.keller@rug.nl}
\urladdr{\url{https://www.timo-keller.de}}

\author[Klagsbrun]{Zev Klagsbrun}
\address{
    Zev Klagsbrun,
    Center for Communications Research - La Jolla, 
    4320 Westerra Court,
    San Diego, CA, 92121}
\email{zdklags@ccr-lajolla.org}

\author[Lowry-Duda]{David Lowry-Duda}
\address{%
  David Lowry-Duda, ICERM, 121 South Main Street, Box E, 11th Floor,
  Providence, RI, 02903
}
\email{\url{david@lowryduda.com}}
\urladdr{\url{https://davidlowryduda.com}}

\author[Morrison]{Travis Morrison}
\address{Travis Morrison, 
Virginia Tech Department of Mathematics, 
226 Stanger Street, 
24061 Blacksburg, VA USA}
\email{\url{tmo@vt.edu}}
\urladdr{\url{https://travismo.github.io/}}

\author[Najman]{Filip Najman}
\address{
Filip Najman, University of Zagreb, Bijeni\v{c}ka Cesta 30, 10000 Zagreb, Croatia
}
\email{\url{fnajman@math.hr}}
\urladdr{\url{https://web.math.pmf.unizg.hr/~fnajman/}}

\author[Shukla]{Himanshu Shukla}
\address{Himanshu Shukla, Mathematisches Institut, Uiversit\"{a}t Bayreuth, Universit\"{a}tstrasse 30, 95444 Bayreuth, Germany}
\email{Himanshu.Shukla@uni-bayreuth.de}
\urladdr{\url{https://www.mathe2.uni-bayreuth.de/hishukla/}}